\documentclass[11pt]{article}
\usepackage{amsmath}
\usepackage{amsfonts}
\usepackage{amssymb}
\usepackage{mathrsfs}
\usepackage{amsthm}
\usepackage{palatino}
\usepackage[margin=2.5cm, vmargin={1.5cm}]{geometry}

\usepackage{dcolumn}

\usepackage{times}
\usepackage{graphicx}
\usepackage{xcolor}

\usepackage{pstricks}
\usepackage{pst-plot}

\renewcommand\footnotemark{}

\begin{document}

\title{Detailed asymptotic expansions for partitions into powers}

\author{
  Cormac ~O'Sullivan\footnote{{\it Date:} Feb 10, 2023.
\newline \indent \ \ \
  \newline \indent \ \ \
  \newline \indent \ \ \
Support for this project was provided by a PSC-CUNY Award, jointly funded by The Professional Staff Congress and The City
\newline \indent \ \ \
University of New York.}
  }

\date{}

\maketitle

\def\s#1#2{\langle \,#1 , #2 \,\rangle}

\def\H{{\mathbf{H}}}
\def\F{{\frak F}}
\def\C{{\mathbb C}}
\def\R{{\mathbb R}}
\def\Z{{\mathbb Z}}
\def\Q{{\mathbb Q}}
\def\N{{\mathbb N}}
\def\G{{\Gamma}}
\def\GH{{\G \backslash \H}}
\def\g{{\gamma}}
\def\L{{\Lambda}}
\def\ee{{\varepsilon}}
\def\K{{\mathcal K}}
\def\Re{\mathrm{Re}}
\def\Im{\mathrm{Im}}
\def\PSL{\mathrm{PSL}}
\def\SL{\mathrm{SL}}
\def\Vol{\operatorname{Vol}}
\def\lqs{\leqslant}
\def\gqs{\geqslant}
\def\sgn{\operatorname{sgn}}
\def\res{\operatornamewithlimits{Res}}
\def\li{\operatorname{Li_2}}
\def\lip{\operatorname{Li}'_2}
\def\pl{\operatorname{Li}}

\def\clp{\operatorname{Cl}'_2}
\def\clpp{\operatorname{Cl}''_2}
\def\farey{\mathscr F}

\def\dm{{\mathcal A}}
\def\ov{{\overline{p}}}
\def\ja{{K}}

\newcommand{\stira}[2]{{\genfrac{[}{]}{0pt}{}{#1}{#2}}}
\newcommand{\stirb}[2]{{\genfrac{\{}{\}}{0pt}{}{#1}{#2}}}
\newcommand{\norm}[1]{\left\lVert #1 \right\rVert}

\newcommand{\e}{\eqref}
\newcommand{\bo}[1]{O\left( #1 \right)}

\newtheorem{theorem}{Theorem}[section]
\newtheorem{lemma}[theorem]{Lemma}
\newtheorem{prop}[theorem]{Proposition}
\newtheorem{conj}[theorem]{Conjecture}
\newtheorem{cor}[theorem]{Corollary}
\newtheorem{assume}[theorem]{Assumptions}
\newtheorem{adef}[theorem]{Definition}


\newcounter{counrem}
\newtheorem{remark}[counrem]{Remark}

\renewcommand{\labelenumi}{(\roman{enumi})}
\newcommand{\spr}[2]{\sideset{}{_{#2}^{-1}}{\textstyle \prod}({#1})}
\newcommand{\spn}[2]{\sideset{}{_{#2}}{\textstyle \prod}({#1})}

\numberwithin{equation}{section}

\let\originalleft\left
\let\originalright\right
\renewcommand{\left}{\mathopen{}\mathclose\bgroup\originalleft}
\renewcommand{\right}{\aftergroup\egroup\originalright}

\bibliographystyle{alpha}

\begin{abstract}
Here we examine the number of ways to partition an integer $n$ into $k$th powers when $n$ is large. Simplified proofs of some asymptotic results of Wright are given using the saddle-point method, including exact formulas for the expansion coefficients. The convexity and log-concavity  of these partitions is shown for large $n$, and the stronger conjectures of Ulas are proved. The asymptotics of Wright's generalized Bessel functions are also treated.
\end{abstract}

\section{Introduction}

Let $k$ be a positive integer and, following Hardy and Ramanujan, write $p^k(n)$ for the number of partitions of $n$ into $k$th powers.  The generating function for $p^k(n)$ is
\begin{equation*}
  G_k(q):=\sum_{n=0}^\infty p^k(n) q^n = \prod_{m=1}^\infty \frac 1{1-q^{m^k}}.
\end{equation*}
To describe the size of $p^k(n)$ as $n\to \infty$, define the useful quantities
\begin{gather}
  c_k:=\tfrac 1k \G(1+\tfrac 1k)\zeta(1+\tfrac 1k), \qquad a_k:=c_k^{k/(k+1)}, \label{quan}\\
  b_k := \frac{a_k}{(2\pi)^{(k+1)/2} (1+\tfrac 1k)^{1/2}}, \qquad h_k:=\frac{\zeta(-k)}2 = -\frac{B_{k+1}}{2(k+1)}, \label{quan2}
\end{gather}
in terms of the gamma function, Riemann zeta function and Bernoulli numbers.
Then $a_k$, $b_k$, $c_k$ are positive real numbers and the first few values of $h_k$ for $k\gqs 1$ are $-\frac 1{24}$, $0$, $\frac 1{240}$ and $0$.
Set
\begin{equation}\label{mkn}
   \mathcal M_k(n) :=  b_k \frac{\exp\left((k+1) a_k (n+h_k)^{1/(k+1)} \right)}{(n+h_k)^{3/2-1/(k+1)}}.
\end{equation}
Hardy and Ramanujan in \cite[p.~111]{HR18} stated the main  term in  the asymptotics
of $p^k(n)$
as $n\to \infty$  and this was extended by Wright \cite[Thm.~2]{Wriii34}:

\begin{theorem}[Wright, 1934] \label{wri}
Let  $k$ and $R$ be positive integers. There exist $\mathcal Q_r(k)$ so that as $n \to \infty$,
\begin{equation} \label{aim}
  p^k(n) = \mathcal M_k(n)\left( 1+\sum_{r=1}^{R-1} \frac{\mathcal Q_r(k)}{(k \cdot a_k)^r (n+h_k)^{r/(k+1)}} + O\left( \frac{1}{n^{R/(k+1)}}\right)\right)
\end{equation}
for an implied constant depending only on $k$ and $R$.
\end{theorem}

There has been renewed interest in this result, with Vaughan \cite{Vau15} for $k=2$, and Gafni \cite{Gaf16} for $k\gqs 2$, giving new treatments of Theorem \ref{wri} using slightly different parameters. Tenenbaum,  Wu and  Li \cite{Tene19} showed that the proof of Theorem \ref{wri}, (expressed in the form \e{aim2x}), can be obtained more easily by employing the saddle-point method.

Wright expressed the  coefficients $\mathcal Q_r(k)$ from \e{aim} in a concise formula that seems to have been overlooked by these recent authors. His asymptotic expansions use the generalized Bessel function $\phi(z)=\phi(\rho,\beta;z)$, see \e{phi2}, which in turn has an expansion that is stated in \cite{Wr33,Wriii34} and finally proved in \cite[Thm. 1]{Wr35}. In our notation the formula is as follows.

\begin{theorem}[Wright, 1935] \label{qkt}
For  all  integers $r \gqs 0$ and  $k \gqs 1$,
\begin{equation}\label{qkr}
  \mathcal Q_r(k) =  r! \binom{-\frac 12}{r} \binom{-\frac 1k}{2}^{-r} \left[ x^{2r}\right] (1+x)^{1/2} \left(
  \frac{(1+x)^{-1/k}-(1-x/k)}{\binom{-\frac 1k}{2} x^2}\right)^{-1/2-r},
\end{equation}
where $[x^{2r}]$ indicates the coefficient of $x^{2r}$ in the succeeding series in $\Q[[x]]$.
\end{theorem}

Note that the series in parentheses on the right of \e{qkr} has initial terms
\begin{equation} \label{sx}
     1- \frac{\tfrac 1k +2}{3} x + \frac{(\tfrac 1k +2)(\tfrac 1k +3)}{3\cdot 4} x^2  - \cdots .
\end{equation}

In Section \ref{ss2} and the beginning of Section \ref{proq} we  prove both Theorems \ref{wri} and \ref{qkt} together. This proof begins in the same way as  \cite{Tene19}, with Propositions \ref{est} and \ref{mell} included for completeness. After that, a  simpler choice of saddle-point allows a different treatment that is much more explicit and that also includes the asymptotic expansion coefficients $\mathcal W_r(\rho,\beta)$ of $\phi(\rho,\beta;z)$.

Various properties of $\mathcal W_r(\rho,\beta)$, and its special case $\mathcal Q_r(k) = \mathcal W_r(\tfrac 1k,-\tfrac 12)$, are  established in Section \ref{proq}.
De Moivre polynomials are a convenient device for manipulating power series, and the
De Moivre polynomial version of \e{qkr} in \e{now} is easy to work with.   For example, $\mathcal Q_0(k)=1$,
\begin{gather}
  \mathcal Q_1(k)=-\frac{11k^2+11k+2}{24(k+1)}, \qquad \mathcal Q_2(k)=-\frac{(k-1)(k+2)(23k^2+23k+2)}{1152(k+1)^2}, \label{qqq}\\
  \mathcal Q_3(k)=-\frac{(k-1)(k+2)(1183k^4+10646k^3+11139k^2+2396k+556)}{414720(k+1)^3},\label{qqq2}
\end{gather}
and we prove some of the patterns that are already appearing.

The polynomials  introduced by De Moivre in 1697 are described briefly as follows and \cite{odm} has more information. Let $n$ and $k$ be integers with $k\gqs 0$. For a power series $a_1 x +a_2 x^2+ a_3 x^3+ \cdots $,  without a constant term and with coefficients in $\C$, we may define the De Moivre polynomial $\dm_{n,k}(a_1, a_2, a_3, \dots)$ by means of the generating function
\begin{equation} \label{bell}
    \left( a_1 x +a_2 x^2+ a_3 x^3+ \cdots \right)^k = \sum_{n\in \Z} \dm_{n,k}(a_1, a_2, a_3, \dots) x^n.
\end{equation}
If $n\gqs k$ then $\dm_{n,k}(a_1, a_2, a_3, \dots)$  is  a polynomial in $a_1, a_2, \dots, a_{n-k+1}$ of homogeneous degree $k$ with positive integer coefficients.
Some simple properties we will need are
\begin{align}\label{gsb}
\dm_{n,k}(0, a_1, a_2, a_3, \dots) & =  \dm_{n-k,k}(a_1, a_2, a_3, \dots), \\
  \dm_{n,k}(c a_1, c a_2, c a_3, \dots) & = c^k \dm_{n,k}(a_1, a_2, a_3, \dots), \label{mulk}\\
  \dm_{n,k}(c a_1, c^2 a_2, c^3 a_3, \dots) & = c^n \dm_{n,k}(a_1, a_2, a_3, \dots). \label{muln}
\end{align}

Section \ref{ss4} gives further results on the asymptotics of $p^k(n)$, including some interesting explicit expansions for the usual partition function $p(n)=p^1(n)$. Proposition \ref{tenk1} shows for instance that
as $n \to \infty$,
\begin{equation*} 
  p(n) = \frac{\exp\left( \pi \sqrt{2n/3} \right)}{4 \sqrt{3} n}\left( 1+\sum_{r=1}^{R-1} \frac{\omega_r}{n^{r/2}} + O\left( \frac{1}{n^{R/2}}\right)\right),
\end{equation*}
 with an explicit determination of the coefficients that seems to be new:
 \begin{equation*} 
  \omega_r = \frac{1}{(-4\sqrt{6})^r} \sum_{k=0}^{(r+1)/2} \binom{r+1}{k} \frac{r+1-k}{(r+1-2k)!}  \left( \frac{\pi}6\right)^{r-2k}.
\end{equation*}

The convexity and log-concavity of the sequence $p^k(n)$  for large $n$ is shown in Section \ref{xv}. Recall that a sequence $a(n)$ is convex at $n$ if $2a(n)\lqs a(n+1)+a(n-1)$  and log-concave there if  $a(n)^2\gqs a(n+1)\cdot a(n-1)$. Ulas conjectured further inequalities in \cite{Ul21} and we prove them here:

\begin{theorem} \label{las}
For each positive integer $k$ there exist $\mathcal C_k$ and $\mathcal D_k$ so that
\begin{alignat}{2}\label{redx}
  2p^k(n) & < \left( p^k(n+1) + p^k(n-1)\right)\left( 1- n^{-2}\right) \qquad && (n\gqs \mathcal C_k),\\
  p^k(n)^2 & > p^k(n+1) \cdot p^k(n-1) \cdot \left( 1+n^{-2}\right) \qquad && (n\gqs \mathcal D_k). \label{redx2}
\end{alignat}
\end{theorem}
In fact Theorem \ref{las} is  stronger than Ulas's conjectures and follows from the sharper results of Corollaries \ref{xyh} and \ref{xyh2}. These in turn follow from a detailed study of Theorem \ref{wri}.

Theorem \ref{big} generalizes Theorem \ref{wri} to give the asymptotics of the generalized
Bessel function $\phi(z)$.
In this last section we also discuss further work of Wright,  (see  \cite[Sect.~4]{And03} as well), including a more precise asymptotic expansion than Theorem \ref{wri}, based on  the circle method and the transformation formulas for $G_k(q)$ he developed.

\section{The asymptotic expansion of $p^k(n)$} \label{ss2}

We give the proofs of Theorem \ref{wri} and most of Theorem \ref{qkt} in this section. In the usual starting point, by Cauchy's theorem, and for any $\sigma>0$,
\begin{align}
  p^k(n) & = \frac 1{2\pi i} \oint G_k(q) q^{-n-1}\, dq \notag\\
  & = \frac 1{2\pi i} \int_{\sigma-i\pi}^{\sigma+i\pi} G_k(e^{-s}) e^{ns}\, ds \label{am2}\\
  & = \frac 1{2\pi} \int_{-\pi}^{\pi} G_k(e^{-\sigma-it}) e^{n(\sigma+it)}\, dt. \label{am}
\end{align}

\subsection{Estimates when $t$ is large}
The next result  will be used to show that small values of $t$ in \e{am} make the main contribution.

\begin{prop} \label{est}
Let $n$, $\lambda$, $\sigma$ and $M$ be  real numbers with $\lambda$, $\sigma$, $M > 0$.  As $\sigma \to 0^+$,
\begin{equation} \label{supw}
  \int_{\lambda \cdot \sigma^{1+1/(3k)}}^{\pi} e^{n it} G_k(e^{-\sigma-it}) \, dt \ll \sigma^M G_k(e^{-\sigma}),
\end{equation}
for an implied constant depending only on $k \in \Z_{\gqs 1}$, $\lambda$ and $M$.
\end{prop}
\begin{proof}
This is a special case of Debruyne and  Tenenbaum's result \cite[Lemma 3.1]{Tene20}, with the minor addition of including an extra parameter $\lambda$. We give their  intricate self-contained proof here, for completeness, showing directly how it applies to $G_k(q)$.
Another approach is shown in \cite[Lemma 2.3]{Tene19}, (with a correction in the final arXiv version).

First let $\Lambda:=\{1^k,2^k,3^k, \dots\}$. It is easy to see that the number of elements of $\Lambda$ of size at most $R\gqs 0$ is $\lfloor R^{1/k}\rfloor \gqs R^{1/k}-1$ and we have
\begin{alignat}{2}\label{ri}
  \left| \Lambda \cap [1,R]\right| & \gqs \tfrac 12 R^{1/k} &\qquad &(R\gqs 2^k), \\
  \left| \Lambda \cap [R,2R]\right| & \gqs (2^{1/k}-1) R^{1/k} &\qquad &(R\gqs 0). \label{ri2}
\end{alignat}
For later use in \e{dir}, choose a real number $R_0$ large enough so that  $(2^{1/k}-1) R_0^{1/k}\gqs 2M$, thinking of $k$ and $M$ as fixed.

Next let $\norm{\vartheta}$ denote the distance from $\vartheta$ to the nearest integer. For every $d>0$ the methods  in \cite[p. 732]{Tene20} establish the inequality
\begin{equation}\label{methw}
  \frac{\left| G_k(e^{-\sigma-it}) \right|}{G_k(e^{-\sigma})}
  \lqs \prod_{m \in \Lambda, \ m\lqs d/\sigma} \left( 1+ \frac{16}{e^d} \cdot
  \frac{\norm{mt/(2\pi)}^2}{m^2 \sigma^2}\right)^{-1/2},
\end{equation}
and this allows us to give detailed bounds for the integrand in \e{supw}.

{\bf First case, when $\lambda \cdot \sigma^{1+1/(3k)} \lqs t \lqs 2\pi \sigma$.}
 Then $0\lqs mt/(2\pi) \lqs 1/2$ when $m\lqs 1/(2\sigma)$ and so with $d=1/2$ in \e{methw},
\begin{equation} \label{get}
  \frac{\left| G_k(e^{-\sigma-it}) \right|}{G_k(e^{-\sigma})}
  \lqs \prod_{m \in \Lambda, \ m\lqs 1/(2\sigma)} \left( 1+
  \frac{\lambda^2 \sigma^{2/(3k)}}{5}\right)^{-1/2}.
\end{equation}
If $\sigma$ satisfies $\sigma \lqs 2^{-k-1}$ then the number of factors on the right side of \e{get} is at least $(2\sigma)^{-1/k}/2$ by \e{ri}. Therefore \e{get} is bounded by
\begin{equation*}
  \exp\left(-\frac 14 \log\left( 1+
  \frac{\lambda^2 \sigma^{2/(3k)}}{5}\right)\frac 1{(2\sigma)^{1/k}} \right) \lqs e^{-C \sigma^{-1/(3k)}} \ll \sigma^M,
\end{equation*}
where $C$ is a positive constant depending only on  $\lambda$. This completes the first case.

The remaining part of the integral has $\sigma<t/(2\pi)\lqs 1/2$.  Dirichlet's approximation theorem lets us find a close rational number $a/q$ to $t/(2\pi)$:
\begin{equation} \label{dir}
  \frac{t}{2\pi} = \frac aq\pm r \quad \text{for} \quad 1\lqs q \lqs 3R_0, \quad (a,q)=1, \quad 0\lqs r \lqs \frac 1{3 R_0 q}.
\end{equation}
Following \cite[Lemma 3.1]{Tene20}, we consider two subcases   separately.

{\bf ``Minor arcs'' when $\sigma<t/(2\pi)\lqs 1/2$ and $2\sigma/(3q)<r\lqs 1/(3R_0 q)$.} Start with \e{methw} for $d=1$. Since $2/(3rq)<1/\sigma$, it follows that
\begin{equation} \label{yt}
  \frac{\left| G_k(e^{-\sigma-it}) \right|}{G_k(e^{-\sigma})}
  \lqs \prod_{\substack{1/(3rq)\lqs m\lqs 2/(3rq) \\ m \in \Lambda}} \left( 1+ 5
  \frac{\norm{mt/(2\pi)}^2}{m^2 \sigma^2}\right)^{-1/2}
  \lqs \prod_{\substack{1/(3rq)\lqs m\lqs 2/(3rq) \\ m \in \Lambda}} \left( 1+
  \frac{5 r^2}{4 \sigma^2}\right)^{-1/2},
\end{equation}
as $\norm{mt/(2\pi)}\gqs 1/(3q)$ here. By \e{ri2} the number of factors in the product is at least $(2^{1/k}-1) (3rq)^{-1/k}$.

If $2\sigma/(3q)<r\lqs \sqrt{\sigma}$ then $3rq \lqs 9 R_0 \sqrt{\sigma}$ and \e{yt} implies
\begin{equation*}
  \frac{\left| G_k(e^{-\sigma-it}) \right|}{G_k(e^{-\sigma})}
  \lqs  \left( 1+
  \frac{5}{9 q^2}\right)^{-C_0 \sigma^{-1/(2k)}} \lqs e^{-C_1 \sigma^{-1/(2k)}} \ll \sigma^M,
\end{equation*}
where $C_0$ and $C_1$ are positive constants depending only on $k$ and $R_0$.

If $\sqrt{\sigma} <r\lqs 1/(3R_0 q)$ then the number of factors in the product on the right of \e{yt} is at least $(2^{1/k}-1) R_0^{1/k}$, and this more than $2M$ by our original choice of $R_0$. Hence we again find
\begin{equation*}
  \frac{\left| G_k(e^{-\sigma-it}) \right|}{G_k(e^{-\sigma})}
  \lqs  \left( 1+
  \frac{5}{4 \sigma}\right)^{-M}  \ll \sigma^M.
\end{equation*}

{\bf ``Major arcs'' when $\sigma<t/(2\pi)\lqs 1/2$ and $0\lqs r \lqs 2\sigma/(3q)$.} Note first that if $q=1$ in \e{dir} then we must have $a=0$ and $t/(2\pi)=r$. This implies $\sigma<r$ and $r \lqs 2\sigma/3$, an impossibility. Therefore, every $t$ in the major arcs has a rational approximation with denominator  $q$ in the range $2\lqs q \lqs 3R_0$. For each of these $q$s we may choose distinct $m_{j,q}$ in $\Lambda$ for $1\lqs j \lqs M$ that are not multiples of $q$. With this choice, let $\sigma$ be small enough that $\sigma\lqs 1/m_{j,q}$ is always true. By \e{methw} with $d=1$,
\begin{equation*}
  \frac{\left| G_k(e^{-\sigma-it}) \right|}{G_k(e^{-\sigma})}
  \lqs \prod_{j=1}^M \left( 1+ 5
  \frac{\norm{m_{j,q}t/(2\pi)}^2}{m_{j,q}^2 \sigma^2}\right)^{-1/2}
  \lqs \prod_{j=1}^M \left( 1+
  \frac{5}{9 q^2 m_{j,q}^2 \sigma^2}\right)^{-1/2} \ll \sigma^M,
\end{equation*}
since $m_{j,q}t/(2\pi)=\ell/q \pm m_{j,q} r$ with $q \nmid \ell$ and $0\lqs m_{j,q} r \lqs 2/(3q)$ implies $\norm{m_{j,q}t/(2\pi)} \gqs 1/(3q)$.
This last case completes the proof of Proposition \ref{est}.
\end{proof}

\subsection{Further setting up}
Put $\Phi_k(s):=\log G_k(e^{-s})$.

\begin{prop} \label{mell}
Let $k$ and $L$ be positive integers. Suppose $s \in \C$ has $\Re(s)>0$ and $|\arg s|\lqs \pi/4$. Then
\begin{equation} \label{rfj}
  \Phi_k(s)= \frac{k c_k}{s^{1/k}}+\frac 12 \log\left( \frac s{(2\pi)^k}\right)+h_k s +\Phi^*_k(s) \qquad \text{where} \qquad \Phi^*_k(s) = O(|s|^L),
\end{equation}
for $c_k$ and $h_k$ in \e{quan}, \e{quan2} and an implied constant depending only on $k$ and $L$.
\end{prop}
\begin{proof}
Now we are following \cite[Lemma 2.1]{Tene19}. The  Mellin transform of  $\Phi_k(s)$ is computed there as $\zeta(z+1)\zeta(kz) \G(z)$, so that
\begin{equation*}
  \Phi_k(s)= \frac 1{2\pi i} \int_{2-i\infty}^{2+i\infty} \zeta(z+1)\zeta(kz) \G(z) \, \frac{dz}{s^z}
  \qquad (\Re(s)>0).
\end{equation*}
 The integrand is meromorphic and for $z=x+i y$ with $|y|$ large and $-L\lqs x\lqs 2$ we have from the usual estimates
\begin{equation*}
  \zeta(z+1)\zeta(kz) \G(z) s^{-z} \ll |y|^A e^{-\pi |y|/2} \left| s^{-x-i y}\right|
\end{equation*}
for $A$ and the implied constant depending only on $k$ and $L$. Also
\begin{equation*}
   \left| s^{-x-i y}\right| = \exp\left(-\Re((x+i y)\log s) \right)
   = \exp\left(-x \log |s| + y \arg s) \right) \lqs |s|^{-x} e^{\pi |y|/4}.
\end{equation*}
Therefore the integrand has exponential decay as $|y| \to \infty$ and moving the contour of integration  from the line with real part $2$ to the line with real part $-L$ is justified. The only poles crossed are at $z=1/k$ and $z=-1$ along with a double pole at $z=0$; the poles of $\G(z)$ at $z=-2,-3,\dots$ cancel with the zeros of $\zeta(z+1)\zeta(kz)$ since $k \in \Z_{\gqs 1}$ and $\zeta(z)$ has zeros at the negative even integers. The  residues of $\zeta(z+1)\zeta(kz) \G(z) s^{-z}$ at $z=1/k$, $0$ and $-1$ are
\begin{equation*}
  \frac{\G(\frac 1k)\zeta(1+\frac 1k)}{k s^{1/k}} = \frac{k c_k}{s^{1/k}}, \qquad \frac 12 \log\left( \frac s{(2\pi)^k}\right), \qquad \frac{\zeta(-k) s}2 = h_k s,
\end{equation*}
respectively, and  \e{rfj} follows.
\end{proof}

Set
\begin{equation*}
 n_k:=n+h_k, \qquad F_k(s):=\frac{k c_k}{s^{1/k}}+n_k s,
\end{equation*}
and the integrand in \e{am2} satisfies
\begin{equation} \label{phhi}
  G_k(e^{-s}) e^{ns} = e^{\Phi_k(s)+ns} = e^{F_k(s)} \frac{s^{1/2}}{(2\pi)^{k/2}}e^{\Phi^*_k(s)}.
\end{equation}
For the saddle-point method, the path of integration in \e{am2} is usually chosen to pass through a point where $\frac d{ds}$ of the integrand is zero. This is the point $\sigma_n$ used in \cite{Tene19}. However, there is some leeway\footnote{In one of the earliest applications of the saddle-point method, Riemann used this flexibility in obtaining asymptotics for $\zeta(s)$ on the critical line; see for example \cite[Sect. 2.1]{OSrs}.} and here we use the more convenient point $s_n$ which is the saddle-point of the factor $e^{F_k(s)}$. In other words,
\begin{equation*}
  s_n=s_{n,k}:=(c_k/n_k)^{k/(k+1)} \quad \implies \quad F'_k(s_n)=0.
\end{equation*}
For large $n$ we see that $s_n$ is a positive real number that approaches $0$ with increasing $n$.

Taking $\sigma = s_n$ and $\lambda=c_k^{-1/3}$ in \e{am} and Proposition \ref{est}, write
\begin{align}\label{uk}
  p^k(n) & =  E_1(n) + \frac 1{2\pi } \int_{-c_k^{-1/3}  s_n^{1+1/(3k)}}^{c_k^{-1/3}  s_n^{1+1/(3k)}} \exp\left( F_k(s_n+it)\right) \frac{(s_n+it)^{1/2}}{(2\pi)^{k/2}} e^{\Phi^*(s_n+i t)} \, dt \\
  \text{for} \quad E_1(n) & := \frac 1{2\pi } \int_{c_k^{-1/3} s_n^{1+1/(3k)}<|t|\lqs \pi} G_k(e^{-s_n-it}) e^{n(s_n+it)} \, dt. \label{newe1}
\end{align}
To make the coming computations  more manageable, apply the change of variables $w=n_k(s_n+i t)$ in \e{uk} to find
\begin{equation*}
 p^k(n) - E_1(n) = (2\pi)^{-k/2} n_k^{-3/2} \frac 1{2\pi i} \int_{U-iU^{2/3}}^{U+iU^{2/3}}
  w^{1/2} \exp\left( w+\frac{k c_k n_k^{1/k}}{w^{1/k}}\right) e^{\Phi^*(w/n_k)} \, dw,
\end{equation*}
where $U=c_k^{k/(k+1)} n_k^{1/(k+1)}$. This integral naturally breaks into the large and small parts
\begin{align}
  I_N & := \frac 1{2\pi i} \int_{U-iU^{2/3}}^{U+iU^{2/3}}
  w^{-\beta} \exp\left( w+\frac{N}{w^{\rho}}\right)  \, dw, \label{ukr}\\
  I'_N & := \frac 1{2\pi i} \int_{U-iU^{2/3}}^{U+iU^{2/3}}
  w^{-\beta} \exp\left( w+\frac{N}{w^{\rho}}\right) \left( e^{\Phi^*(w/n_k)}-1 \right) \, dw,\label{ukr2}
\end{align}
respectively, for $\beta=-1/2$ and  the simpler variables
\begin{equation*}
  N=k c_k n_k^{1/k}, \qquad \rho=1/k.
\end{equation*}
Therefore, in this notation,
\begin{equation}\label{noty}
  p^k(n) =  E_1(n)  + (2\pi)^{-k/2} n_k^{-3/2}\left( I_{N}+I'_{N}\right) .
\end{equation}

\subsection{The asymptotics of $I_N$}
Now we focus on $I_N$  where we may fix $\rho \in \R_{>0}$, $\beta\in \C$  and let real $N$ tend to infinity. In terms of these variables,
\begin{equation} \label{kyv}
  U=(\rho N)^{1/(\rho+1)} \qquad \text{so that} \qquad N U^{-\rho} = U/\rho.
\end{equation}
 Also write
\begin{gather} \label{krz}
 f_N(w):= w+\frac{N}{w^{\rho}} \qquad \text{with} \qquad \frac{f_N^{(j)}(w)}{j!} = \binom{-\rho}{j} \frac{N}{w^{\rho+j}} \qquad (j \gqs 2),\\
 I_N = \frac 1{2\pi } \int_{-U^{2/3}}^{U^{2/3}}
   \exp\left( f_N(U+it)\right) (U+it)^{-\beta} \, dt. \notag
\end{gather}

Assume from here on that $N$ is large enough to ensure that $U>1$. Then $f_N(z)$ is holomorphic in a domain containing the disk of radius $U^{2/3}$ about $U$. Expanding $f_N(z)$ in its Taylor series there shows
\begin{equation} \label{co}
  I_N = \frac{e^{f_N(U)}}{2\pi U^\beta} \int_{-U^{2/3}}^{U^{2/3}}
    \exp\left( \frac{f''_N(U)}{2}(it)^2\right) g_N(t)\, dt,
\end{equation}
where $g_N(z)$ is holomorphic on the same disk and given by
\begin{equation*}
  g_N(z)= \exp\left( \sum_{j=3}^\infty \frac{f_N^{(j)}(U)}{j!}(iz)^j\right) \cdot \left(1+\frac{i z}{U} \right)^{-\beta}.
\end{equation*}

\begin{lemma} \label{sup}
Suppose $z\in \C$ with $|z|\lqs U^{2/3}$. Then for these $z$ values and $N$ large enough we have the bound $g_N(z) \ll 1$ and the development
\begin{equation} \label{gnz}
  g_N(z)= \sum_{j=0}^\infty d_j \left( \frac{i z}{U}\right)^j
\end{equation}
for
\begin{equation} \label{dj}
 d_j= \sum_{\ell=0}^{j/3} \frac{U^\ell}{\rho^\ell \ell!}
  \sum_{m=3\ell}^j \binom{-\beta}{j-m} \dm_{m-2\ell,\ell}\left(\binom{-\rho}{3}, \binom{-\rho}{4},  \dots \right).
  \end{equation}
\end{lemma}
\begin{proof}
To bound $g_N(z)$ when $|z|\lqs U^{2/3}$, first note with \e{krz} that
\begin{equation*}
  \sum_{j=3}^\infty \frac{f_N^{(j)}(U)}{j!}(iz)^j \ll \sum_{j=3}^\infty \left| \binom{-\rho}{j} \right|
  \frac{N}{U^{\rho+j}} U^{2j/3}
  = \frac 1{\rho}\sum_{j=3}^\infty  \binom{-\rho}{j} (-1)^j
  U^{1-j/3}.
\end{equation*}
As $1-j/3 \lqs -j/6$ for $j\gqs 6$, we have
\begin{equation*}
  \sum_{j=3}^\infty \frac{f_N^{(j)}(U)}{j!}(iz)^j
   \ll 1+ \sum_{j=6}^\infty  \binom{-\rho}{j} (-1)^j
  U^{-j/6} \lqs \left(1-U^{-1/6} \right)^{-\rho} \ll 1,
  \end{equation*}
and hence $g_N(z) \ll 1$ as we claimed.
To find $d_j$ write
\begin{align*}
  \exp\left( \sum_{j=3}^\infty \frac{f_N^{(j)}(U)}{j!}(iz)^j\right) &
  = \sum_{m=0}^\infty (i z)^m \sum_{\ell=0}^m \frac 1{\ell!} \dm_{m,\ell}\left(0,0,\frac{f_N^{(3)}(U)}{3!},\frac{f_N^{(4)}(U)}{4!},\dots \right) \\
   & =\sum_{m=0}^\infty \left( \frac{i z}{U}\right)^m
   \sum_{\ell=0}^{m/3} \frac{N^\ell}{\ell!}
  U^{-\rho\ell} \dm_{m-2\ell,\ell}\left(\binom{-\rho}{3}, \binom{-\rho}{4},  \dots \right),
\end{align*}
where we used \e{gsb}, \e{mulk} and \e{muln}.
Developing $\left(1+i z/U \right)^{-\beta}$ with the binomial theorem and interchanging summations completes the proof of \e{dj}.
\end{proof}

\begin{lemma} \label{sup2}
Suppose $J$ is a positive integer and $z\in \C$ with $|z|\lqs U^{2/3}$. Then for $N$ large enough
\begin{equation} \label{gnz2}
  g_N(z)= \sum_{j=0}^{J-1} d_j \left( \frac{i z}{U}\right)^j + g_{N,J}(z) \qquad \text{with} \qquad g_{n,J}(z) = O\left( \frac{|z|^J}{U^{2J/3}}\right)
\end{equation}
and an implied constant  independent of $N$.
\end{lemma}
\begin{proof}
The same proof as for Lemma \ref{sup} demonstrates that $g_N(z) \ll 1$ for $|z|\lqs R:=2U^{2/3}$.
By Taylor's theorem,  \cite[pp. 125-126]{Al},
\begin{equation} \label{gar}
  g_N(z) - \sum_{j=0}^{J-1} d_j \left( \frac{i z}{U}\right)^j = \frac{z^J}{2\pi i} \int_{|w|=R} \frac{g_N(w)}{w^J (w-z)} \, dw.
\end{equation}
Since $|w-z|\gqs R/2$, it follows that the right side of \e{gar} is $\ll |z/R|^J$ as we wanted.
\end{proof}


With \e{kyv} and \e{krz}, set
\begin{equation} \label{end}
  Y:= \frac{f_N''(U)}{2} = \binom{-\rho}{2} \frac{N}{U^{\rho+2}} = \binom{-\rho}{2}\frac{1}{\rho U}.
\end{equation}
For any positive integer $J$ write $I_N=I_N^{*}-E_2(N)+E_3(N)$ with
\begin{align}
  I_N^{*} & := \frac{e^{f_N(U)}}{2\pi U^\beta} \int_{-\infty}^{\infty} \exp\left( -Y \cdot t^2\right) \sum_{j=0}^{2J-1} d_j \left( \frac{i t}{U}\right)^j\, dt, \label{pss}\\
  E_2(N) & := \frac{e^{f_N(U)}}{2\pi U^\beta} \int_{|t|>U^{2/3}} \exp\left( -Y \cdot t^2\right) \sum_{j=0}^{2J-1} d_j \left( \frac{i t}{U}\right)^j\, dt,\label{e3}\\
  E_3(N) & := \frac{e^{f_N(U)}}{2\pi U^\beta} \int_{-U^{2/3}}^{U^{2/3}} \exp\left( -Y \cdot t^2\right) g_{N,2J}(t)\, dt. \label{e4}
\end{align}
Only the even values of $j$ give  nonzero contributions in \e{pss} and \e{e3}, so $j$ may be replaced by $2j$ in the summands.  The well-known identity
\begin{equation} \label{et2}
  \int_{-\infty}^\infty e^{-Y t^2} t^{2j}\, dt = \sqrt{\pi} j! \binom{-\frac 12}{j} \frac{(-1)^j}{Y^{j+1/2}}
  \qquad (Y>0, j\in \Z_{\gqs 0}),
\end{equation}
 implies
\begin{equation*}
  I_N^{*} = \frac{e^{f_N(U)}}{2\pi U^\beta}  \sum_{j=0}^{J-1} \sqrt{\pi} j! \binom{-\frac 12}{j} \frac{d_{2j}}{U^{2j} Y^{j+1/2}}.
\end{equation*}
Then the relations
\begin{equation} \label{fnu}
   f_N(U)=(1+\tfrac 1{\rho})U, \qquad \frac{1}{U^{2j} Y^{j+1/2}} = \left(\frac{2U}{\rho+1}\right)^{1/2} \binom{-\rho}{2}^{-j} \frac{\rho^j}{U^j},
\end{equation}
show
\begin{equation} \label{is}
   I_N^{*} = \frac{ U^{1/2-\beta}e^{(1+\tfrac 1{\rho})U}}{\sqrt{2\pi (\rho+1)}}  \sum_{j=0}^{J-1}  j!
    \binom{-\frac 12}{j} \binom{-\rho}{2}^{-j}\frac{\rho^j }{U^{j}}d_{2j}.
\end{equation}
With \e{dj} the sum in \e{is} is
\begin{equation*}
  \sum_{j=0}^{J-1} j! \binom{-\frac 12}{j} \binom{-\rho}{2}^{-j}
  \sum_{\ell=0}^{2j/3} \frac{\rho^{j-\ell}}{ U^{j-\ell} \ell!}
  \sum_{m=3\ell}^{2j} \binom{-\beta}{2j-m} \dm_{m-2\ell,\ell},
\end{equation*}
and writing $r=j-\ell$, this equals
\begin{equation} \label{cmd}
  \sum_{r=0}^{J-1}  \frac{\rho^r}{U^{r}} \sum_{\ell=0}^{2r} \binom{r+\ell}{r} \binom{-\frac 12}{r+\ell} \binom{-\rho}{2}^{-r-\ell}r!
  \sum_{m=3\ell}^{2r+2\ell} \binom{-\beta}{2r+2\ell-m} \dm_{m-2\ell,\ell}.
\end{equation}
Using
\begin{equation*}
  \binom{r+\ell}{r} \binom{-\frac 12}{r+\ell} = \binom{-\frac 12}{r} \binom{-\frac 12-r}{\ell}
\end{equation*}
along with the substitution $v=2r+2\ell-m$ makes the inner sums in \e{cmd} into
\begin{equation} \label{qkr2}
  \mathcal W_r(\rho,\beta):= r! \binom{-\frac 12}{r}\sum_{\ell=0}^{2r} \binom{-\frac 12-r}{\ell}\binom{-\rho}{2}^{-r-\ell}
  \sum_{v=0}^{2r} \binom{-\beta}{v} \dm_{2r-v,\ell}\left(\binom{-\rho}{3}, \binom{-\rho}{4},  \dots \right),
\end{equation}
so that finally,
\begin{equation} \label{aim2}
 I_N^{*} = \frac{U^{1/2-\beta}}{\sqrt{2\pi (\rho+1)}}
  \exp\left((1+\tfrac 1{\rho})U \right)
 \left( 1+\sum_{r=1}^{J-1} \frac{\rho^r \mathcal W_r(\rho,\beta)}{U^{r}}
 \right).
\end{equation}
This is our main term and the next result is demonstrated  by showing that $E_2$ and $E_3$ are smaller.

\begin{prop} \label{inz}
Fix $\rho>0$ and  $\beta \in \C$. Set $U:=(\rho N)^{1/(\rho+1)}$ and
\begin{equation*}
  I_N  := \frac 1{2\pi i} \int_{U-iU^{2/3}}^{U+iU^{2/3}}
  w^{-\beta} \exp\left( w+\frac{N}{w^{\rho}}\right)  \, dw
\end{equation*}
as in \e{ukr}. Then as real $N \to \infty$,
\begin{equation} \label{want}
  I_N = \frac{U^{1/2-\beta}}{\sqrt{2\pi (\rho+1)}}
  \exp\left((1+\tfrac 1{\rho})U \right)
 \left( 1+\sum_{r=1}^{R-1} \frac{\rho^r \mathcal W_r(\rho,\beta)}{U^{r}}
 +\bo{\frac{1}{ U^{R}}}\right),
\end{equation}
for an implied constant depending only on $R \in \Z_{\gqs 1}$, $\rho$ and $\beta$.
\end{prop}
\begin{proof}
We begin by bounding $E_2(N)$ in \e{e3}; the parameter $J$ will be chosen later. For all $a,r \gqs 0$ and $c>0$
\begin{equation}\label{gmmb}
  \int_a^\infty e^{-c x} x^{r}\, dx \ll c^{-r-1} \left((a c)^r +1\right) e^{-a c}
\end{equation}
for an implied constant depending only on $r$. This elementary bound is shown in \cite[Eq. (2.8)]{OSxi} for example.
So the integral in \e{e3} satisfies
\begin{align*}
  \int_{U^{2/3}}^\infty \exp\left( -Y \cdot t^2\right) t^{2j}\, dt &
   = \frac 12 \int_{U^{4/3}}^\infty \exp\left( -Y \cdot x\right) x^{j-1/2}\, dx\\
   & \ll Y^{-j-1/2} \left((U^{4/3}Y )^{j-1/2} +1\right) \exp\left(- U^{4/3}Y \right)\\
   & \ll U^{4j/3+1/2} \exp\left(-\tfrac{\rho+1}2 U^{1/3} \right).
\end{align*}
Since $d_{2j} \ll U^{2j/3}$ is easily seen from \e{dj}, we obtain
\begin{multline} \label{iik}
  E_2(N) \ll \frac{1}{|U^{\beta}|}
  \exp\left((1+\tfrac 1{\rho})U \right) \sum_{j=0}^{J-1} \frac{d_{2j}}{U^{2j}}  U^{4j/3+1/2} \exp\left(- \tfrac{\rho+1}2 U^{1/3}\right)\\
  \ll \frac{U^{1/2}}{|U^{\beta}|}
  \exp\left((1+\tfrac 1{\rho})U \right)  \exp\left(- \tfrac{\rho+1}2 U^{1/3}\right).
\end{multline}
Next we bound  $E_3(N)$ in \e{e4}. With  Lemma \ref{sup2} and \e{et2},
\begin{multline} \label{iik2}
  E_3(N) \ll \frac{1}{|U^{\beta}|}
  \exp\left((1+\tfrac 1{\rho})U \right) \int_{-\infty}^\infty e^{-Y t^2} \frac{t^{2J}}{U^{4J/3}} \,dt\\
  \ll \frac{1}{|U^{\beta}|}
  \exp\left((1+\tfrac 1{\rho})U \right) \frac{1}{Y^{J+1/2}}
  \ll \frac{U^{1/2}}{|U^{\beta}|}
  \exp\left((1+\tfrac 1{\rho})U \right) \frac{1}{U^{J/3}}.
\end{multline}
As $I_N=I_N^{*}-E_2(N)+E_3(N)$, it follows from \e{aim2}, \e{iik} and \e{iik2} that
\begin{equation*}
  I_N = \frac{U^{1/2-\beta}}{\sqrt{2\pi (\rho+1)}}
  \exp\left((1+\tfrac 1{\rho})U \right)
 \left( 1+\sum_{r=1}^{J-1} \frac{\rho^r \mathcal W_r(\rho,\beta)}{U^{r}}
 +\bo{\exp\left(- \tfrac{\rho+1}2 U^{1/3}\right) + \frac{1}{U^{J/3}}}\right).
\end{equation*}
Choosing $J=3R$ then completes the proof.
\end{proof}

\subsection{Final steps}
The companion integral $I'_N$ in \e{ukr2} to $I_N$ in \e{ukr} is supposed to be smaller. This is proved next.

\begin{prop}  \label{inz2}
Fix $\rho$, $x>0$ and  $\beta \in \C$. With $N>0$ set $U:=(\rho N)^{1/(\rho+1)}$ and
\begin{equation*}
  I'_N  := \frac 1{2\pi} \int_{-U^{2/3}}^{U^{2/3}}
  \exp\left( f_N(U+it)\right) (U+it)^{-\beta}  \left( e^{\Phi^*((U+it)/x)}-1 \right) \, dt
\end{equation*}
as in \e{ukr2}, (with $x$ replacing $n_k$). Assume $x>U$. Then for any $L>0$,
\begin{equation} \label{want2}
  I'_N \ll \frac{U^{1/2}}{|U^{\beta}|}
  \exp\left((1+\tfrac 1{\rho})U \right) \frac{U^L}{x^L},
\end{equation}
as real $N \to \infty$, for an implied constant depending only on $L$, $\rho$ and $\beta$.
\end{prop}
\begin{proof}
By Proposition \ref{mell} we know $\Phi^*((U+it)/x) \ll U^L/x^L$ for $|t|\lqs U$. Hence, for $U/x<1$,
\begin{equation*}
  e^{\Phi^*((U+it)/x)}-1 \ll U^L/x^L.
\end{equation*}
As in \e{co}, using $g_N(z)\ll 1$ from Lemma \ref{sup},
\begin{align*}
  I'_N & = \frac{e^{f_N(U)}}{2\pi U^\beta} \int_{-U^{2/3}}^{U^{2/3}}
    \exp\left( \frac{f''_N(U)}{2}(it)^2\right) g_N(t) \left( e^{\Phi^*((U+it)/x)}-1 \right)\, dt \\
   & \ll \frac{e^{f_N(U)}}{|U^\beta|} \int_{-U^{2/3}}^{U^{2/3}}
    \exp\left( -Y t^2\right)   \frac{U^L}{x^L}\, dt\\
    & < \frac{e^{f_N(U)}}{|U^\beta|} \frac{U^L}{x^L} \int_{-\infty}^{\infty}
    \exp\left( -Y t^2\right)   \, dt,
\end{align*}
and \e{want2} follows by employing \e{end}, \e{et2} and the left identity in \e{fnu}.
\end{proof}

\begin{proof}[Proof of Theorem \ref{wri}]
Recall \e{quan}, \e{quan2}, \e{mkn}, and make the substitutions, ($n_k:=n+h_k$),
\begin{equation} \label{ori}
  N=k c_k n_k^{1/k}, \qquad \rho=1/k, \qquad \beta=-1/2, \qquad U=a_k n_k^{1/(k+1)}, \qquad x=n_k,
\end{equation}
in Propositions \ref{inz} and \ref{inz2}, going back to our original variables. Then, after choosing $L$ large enough in \e{want2} that $L k \gqs R$,  \e{noty} implies
\begin{equation} \label{toerr}
   p^k(n) -E_1(n) = \mathcal M_k(n)\left( 1+\sum_{r=1}^{R-1} \frac{\mathcal W_r(\tfrac 1k,-\tfrac 12)}{(k \cdot a_k)^r n_k^{r/(k+1)}} + O\left( \frac{1}{n^{R/(k+1)}}\right)\right),
\end{equation}
as $n \to \infty$, with $\mathcal Q_r(k):=\mathcal W_r(\tfrac 1k,-\tfrac 12)$ given in \e{qkr2}.

Finally, we claim that $E_1(n)$ may be moved inside the error term in \e{toerr}.
Applying Proposition \ref{est} to \e{newe1} shows
\begin{align*}
   E_1(n)
  & \ll e^{n s_n} s_n^M P(e^{-s_n})
   \ll s_n^M e^{F_k(s_n)}s_n^{1/2} e^{\Phi^*(s_n)} \ll s_n^{M+1/2} e^{F_k(s_n)},
\end{align*}
where we also needed  \e{phhi} and the bound in \e{rfj}. Use the equality $F_k(s_n)=(k+1)a_k n_k^{1/(k+1)}$ to find
\begin{equation*}
  E_1(n) \ll \mathcal M_k(n) \cdot n_k^{-(M-2)k/(k+1)},
\end{equation*}
and choosing any $M$ satisfying $(M-2)k\gqs R$ proves the claim.
This completes the proof of Theorem \ref{wri} while also giving the explicit formula $\mathcal W_r(\tfrac 1k,-\tfrac 12)$ from \e{qkr2} for the expansion coefficients $\mathcal Q_r(k)$.
\end{proof}

\section{Properties of $\mathcal Q_r(k)$ and $\mathcal W_r(\rho,\beta)$} \label{proq}

In this section we develop some properties of $\mathcal W_r(\rho,\beta)$ for $\rho>0$ and $\beta \in \C$. These quantities will be required for the asymptotic expansion of Wright's generalized Bessel function in Section \ref{conw}. As seen in \e{toerr}, the special case $\mathcal W_r(\tfrac 1k,-\tfrac 12)$ equals $\mathcal Q_r(k)$ which is needed in the asymptotic expansion of $p^k(n)$.

Define  the series in $\Q[\rho][[x]]$
\begin{align*}
  S(\rho,x):=\frac{(1+x)^{-\rho}-1-\binom{-\rho}{1} x}{\binom{-\rho}{2} x^2} & = 1+ \frac{\binom{-\rho}{3}}{\binom{-\rho}{2}} x+  \frac{\binom{-\rho}{4}}{\binom{-\rho}{2}} x^2+ \cdots \\
   & = 1- \frac{\rho+2}{3} x + \frac{(\rho+2)(\rho+3)}{3\cdot 4} x^2 - \cdots .
\end{align*}

\begin{theorem} \label{rss}
For  all  integers $r \gqs 0$, the coefficients $\mathcal W_r(\rho,\beta)$ defined in \e{qkr2} satisfy
\begin{equation}\label{ed}
  \mathcal W_r(\rho,\beta) =  r! \binom{-\frac 12}{r} \binom{-\rho}{2}^{-r} \left[ x^{2r}\right] (1+x)^{-\beta} S(\rho,x)^{-1/2-r},
\end{equation}
where again $[x^{2r}]$ indicates the coefficient of $x^{2r}$ in the succeeding series.
\end{theorem}
\begin{proof}
With $\sum_{v=0}^\infty \binom{-\beta}{v} x^v  = (1+x)^{-\beta}$ and
\begin{equation*}
  \sum_{v=0}^\infty \dm_{v,\ell}\left(\binom{-\rho}{3}, \binom{-\rho}{4},  \dots \right) x^v  = \left( \binom{-\rho}{3}x+ \binom{-\rho}{4}x^2 +  \cdots \right)^{\ell},
\end{equation*}
we see that the inner sum in \e{qkr2} equals
\begin{multline*}
  \left[ x^{2r}\right] (1+x)^{-\beta} \left(\left(
  (1+x)^{-\rho}-1-\binom{-\rho}{1} x - \binom{-\rho}{2} x^2\right)/x^2\right)^\ell\\
  = \left[ x^{2r}\right] (1+x)^{-\beta} \left(
  S(\rho,x)-1\right)^\ell \binom{-\rho}{2}^\ell.
\end{multline*}
Hence,
\begin{align*}
  \mathcal W_r(\rho,\beta) & = r! \binom{-\frac 12}{r}\sum_{\ell=0}^{2r} \binom{-\frac 12-r}{\ell}\binom{-\rho}{2}^{-r-\ell} \left[ x^{2r}\right] (1+x)^{-\beta} \left(
  S(\rho, x)-1\right)^\ell \binom{-\rho}{2}^\ell \\
   & = r! \binom{-\frac 12}{r} \binom{-\rho}{2}^{-r} \left[ x^{2r}\right] \sum_{\ell=0}^{2r} \binom{-\frac 12-r}{\ell}(1+x)^{-\beta} \left(
  S(\rho, x)-1\right)^\ell \\
  & = r! \binom{-\frac 12}{r} \binom{-\rho}{2}^{-r} \left[ x^{2r}\right] (1+x)^{-\beta}
  S(\rho, x)^{-1/2-r},
\end{align*}
as we wanted to show.
\end{proof}

Then Theorem  \ref{qkt} is the special case of Theorem \ref{rss} with $\rho=1/k$ and $\beta=-1/2$.
Also by \e{qkr2},
\begin{equation} \label{now}
  \mathcal Q_r(k) = r! \binom{-\frac 12}{r}\sum_{\ell=0}^{2r} \binom{-\frac 12-r}{\ell}\binom{-\frac 1k}{2}^{-r-\ell}
  \sum_{v=0}^{2r} \binom{\frac 12}{v} \dm_{2r-v,\ell}\left(\binom{-\frac 1k}{3}, \binom{-\frac 1k}{4},  \dots \right).
\end{equation}

\begin{lemma} \label{wpoly}
We have
\begin{equation*}
  \mathcal W_r(\rho,\beta) = \frac{\mathcal V_r(\rho,\beta)}{\rho^r(\rho+1)^r},
\end{equation*}
 for $\mathcal V_r(\rho,\beta)$ a polynomial of degree at most $2r$ in $\Q[\rho,\beta]$.
\end{lemma}
\begin{proof}
Rewrite the binomial coefficients in \e{qkr2} as
\begin{equation*}
  \binom{-\rho}{j+2} = \rho(\rho +1)(-1)^j  w_j \quad \text{for} \quad w_j:=\frac{(\rho+2)(\rho+3) \cdots (\rho+j+1)}{(j+2)!}
\end{equation*}
when $j\gqs 0$, (with $w_0=1/2$). Then by \e{mulk}, \e{muln},
\begin{equation*}
  \dm_{v,\ell}\left(\binom{-\rho}{3}, \binom{-\rho}{4},  \dots \right)
  = (-1)^v(\rho(\rho+1))^\ell \dm_{v,\ell}\left(w_1, w_2, w_3,  \dots \right),
\end{equation*}
and inserting this into \e{qkr2} and simplifying shows
\begin{equation} \label{qkr3}
  \mathcal W_r(\rho,\beta)=  \binom{-\frac 12}{r} \frac{r!}{\rho^r(\rho+1)^r} \sum_{\ell=0}^{2r} \binom{-\frac 12-r}{\ell}
  2^{r+\ell} \sum_{v=0}^{2r} \binom{-\beta}{v} (-1)^v \dm_{2r-v,\ell}\left(w_1, w_2, w_3,  \dots \right).
\end{equation}
Since $w_j$ is a degree $j$ polynomial in $\rho$, it follows from \e{muln} that $\dm_{2r-v,\ell}\left(w_1, w_2, w_3,  \dots \right)$ is a degree $2r-v$ polynomial in $\rho$, while $\binom{-\beta}{v}$ is a degree $v$ polynomial in $\beta$.
\end{proof}

\begin{cor} \label{poly}
There exist $\mathcal P_r(x)$ in $\Q[x]$, of degree at most $2r$, so that
\begin{equation*}
  \mathcal Q_r(k) = \frac{\mathcal P_r(k)}{(k+1)^r}.
\end{equation*}
\end{cor}

With Corollary \ref{poly} we see that the definition of  $\mathcal Q_r(k)$ extends from $k \in \Z_{\gqs 1}$ to $k$ being any number except $-1$. Examples of $\mathcal P_r(x)$ for $r=1$, $2$, $3$ may be seen in \e{qqq}, \e{qqq2}. Numerically examining the zeros of $\mathcal P_r(x)$ for $r\lqs 20$ we see they are mostly real, with a small number of complex roots slightly to the left of the imaginary axis. By the rational roots test, the only rational roots for $r\lqs 20$ are $x=1$ and $x=-2$.

\begin{prop} \label{1-2}
For $r\gqs 2$ the polynomial $\mathcal P_r(x)$ has roots $x=1$ and $x=-2$.
\end{prop}
\begin{proof}
When  $k$ equals $1$ we have
\begin{align*}
  \mathcal Q_r(1) & =  r! \binom{-\frac 12}{r} \binom{-1}{2}^{-r} \left[ x^{2r}\right] (1+x)^{1/2} \left(
  \frac{(1+x)^{-1}-(1-x)}{\binom{-1}{2} x^2}\right)^{-1/2-r}  \\
   &  =  r! \binom{-\frac 12}{r}  \left[ x^{2r}\right] (1+x)^{r+1}
\end{align*}
by \e{qkr} and so $\mathcal Q_r(1)$ and $\mathcal P_r(1)$ equal $0$ for $r\gqs 2$.

When  $k$ equals $-2$ we have
\begin{equation} \label{any}
  \mathcal Q_r(-2) =  r! \binom{-\frac 12}{r} \left(-\frac 18 \right)^{-r} \left[ x^{2r}\right] (1+x)^{1/2} \left(
  \frac{(1+x)^{1/2}-1-x/2}{-\frac 18  x^2}\right)^{-1/2-r}.
\end{equation}
Let $y:= (1+x)^{1/2}$, a  series in $x$ with coefficients $\binom{1/2}{j}$. Then $x=y^2-1$ and we find that the series to the right of $\left[ x^{2r}\right]$ in \e{any} equals
\begin{equation*}
  y\left(-8y+8+4x \right)^{-1/2-r} x^{2r+1} = y(1+y)^{2r+1}/2^{2r+1}.
\end{equation*}

Now we claim, for positive integers $m$, that $(1+y)^m$, when considered as a series in $x$, has zero coefficients for $x^j$ when $\lfloor m/2\rfloor+1 \lqs j \lqs m-1$. If we notice that
\begin{equation*}
  (1+y)^m + (1-y)^m = 2 \sum_{j=0}^{m/2} \binom{m}{2j} y^{2j} = 2 \sum_{j=0}^{m/2} \binom{m}{2j} (1+x)^{j}
\end{equation*}
and
\begin{equation*}
   (1-y)^m = \left(1-1- \binom{\frac 12}{1}x- \binom{\frac 12}{2}x^2- \cdots \right)^m = x^{m} \left(- \binom{\frac 12}{1}- \binom{\frac 12}{2}x- \cdots \right)^m,
\end{equation*}
then it is clear that $(1+y)^m$ is the sum of a polynomial in $x$ of degree at most $\lfloor m/2\rfloor$ and a series whose lowest degree term has degree $m$. This proves our claim.

It follows that $y(y+1)^{2r+1} = (y+1)^{2r+2}-(y+1)^{2r+1}$ has zero coefficients for $x^j$ when $r+2\lqs j \lqs 2r$ and in particular $\mathcal Q_r(-2)$ and $\mathcal P_r(-2)$ equal $0$ for $r\gqs 2$.
\end{proof}

The above claim may also be proved by means of the generalized binomial series $\mathcal B_{-1}(x)$ from \cite[p. 203]{Knu}. We have
\begin{equation} \label{genb}
  \left(\frac{1+y}2\right)^m = \left( \frac{1+(1+x)^{1/2}}2\right)^m = \mathcal B_{-1}(x/4)^m=
  1+\sum_{j=1}^\infty \frac mj \binom{m-j-1}{j-1} \frac {x^j}{4^j}
\end{equation}
and the binomial coefficient in \e{genb} is zero for $\lfloor m/2\rfloor+1 \lqs j \lqs m-1$.
See \cite[Sects. 8, 9]{OSsym} where \e{genb} and generalizations   are proved using Lagrange inversion. Note that \e{genb} is an identity for formal power series and valid for $m$ a variable or any complex number.

If $k=1$ then $p^1(n)$ is the usual partition function $p(n)$. We have seen with Proposition \ref{1-2} that $\mathcal Q_r(1)=0$ for $r\gqs 2$, and therefore Theorem \ref{wri} implies:
\begin{theorem} \label{radq}
Let $n$ and $R$ be positive integers. As $n \to \infty$,
\begin{equation} \label{vuq}
  p(n) = \frac{\exp\left( \pi \sqrt{\tfrac 23(n-\tfrac 1{24})} \right)}{4 \sqrt{3} (n-\tfrac 1{24})}\left( 1- \frac{1}{\pi \sqrt{\tfrac 23(n-\tfrac 1{24})}} + O\left( \frac{1}{n^{R/2}}\right)\right),
\end{equation}
with an implied constant depending only on  $R$.
\end{theorem}

This of course also follows from the stronger results of Hardy and Ramanujan or Rademacher, taking just the first term in their asymptotic expansions. From \cite[p. 278]{Ra}:
\begin{equation} \label{vuqr}
  p(n) = \frac{\exp\left( \pi \sqrt{\tfrac 23(n-\tfrac 1{24})} \right)}{4 \sqrt{3} (n-\tfrac 1{24})}\left( 1- \frac{1}{\pi \sqrt{\tfrac 23(n-\tfrac 1{24})}} + O\left( \exp\left(-\frac{\pi}2 \sqrt{\tfrac 23(n-\tfrac 1{24})} \right)\right)\right).
\end{equation}
In the $k=-2$ case,  $\mathcal Q_r(-2)=0$ for $r\gqs 2$ by Corollary \ref{poly} and Proposition \ref{1-2}. Then it is interesting to speculate that  there are similar  expansions to \e{vuq} and \e{vuqr} for a function ``$p^{-2}(n)$''.

Lastly in this section we give an explicit formula for $\mathcal W_r(\rho,\beta)$ that involves only binomial and multinomial coefficients.

\begin{prop} \label{wrfo}
For $r\gqs 0$,
\begin{multline*}
  \mathcal W_r(\rho,\beta)=  r!\binom{-\frac 12}{r} \sum_{\ell=0}^{2r} \binom{3r+1/2}{2r-\ell}\left(\frac {2}{\rho(\rho+1)} \right)^{r+\ell} \\
  \times \sum_{j_1+j_2+j_3 =\ell} \binom{-1/2-r}{j_1,j_2,j_3}
  \binom{-j_3 \rho -\beta}{2r+2\ell-j_2}(-1)^{j_1} \rho^{j_2}.
\end{multline*}
\end{prop}

To find a simpler formula for the De Moivre polynomial in \e{qkr2}, we first describe a general technique for dealing with shifted coefficients.
Applying the straightforward identity
\begin{equation*}
  \dm_{m,\ell}(a_2,a_3, \dots)  = \sum_{j=0}^\ell (-a_1)^{\ell-j} \binom{\ell}{j} \dm_{m+j,j}(a_1,a_2, \dots), 
\end{equation*}
$r$ times shows:

\begin{prop} \label{ref}
For $r, \ell \gqs 0$ we have that
$\dm_{m,\ell}(a_{r+1},a_{r+2}, \dots)$ equals
\begin{equation}\label{fortx}
   \sum_{ j_1+ j_2+ \dots +j_{r+1}= \ell}
 \binom{\ell}{j_1 , j_2 ,  \dots , j_{r+1}} (-a_1)^{j_1} (-a_2)^{j_2}  \cdots (-a_r)^{j_r}
 \dm_{m+J+r j_{r+1},j_{r+1}}(a_{1},a_{2}, \dots)
\end{equation}
where $J$ means $(r-1)j_1+(r-2)j_2+ \cdots +1 j_{r-1}$ and the summation is   over all  $j_1$,  \dots , $j_{r+1} \in \Z_{\gqs 0}$ with sum $\ell$.
\end{prop}

Noting that
\begin{equation*}
  \dm_{m,\ell}\left( \binom{z}{0}, \binom{z}{1}, \binom{z}{2},\dots  \right)  = \binom{\ell z}{m-\ell},
\end{equation*}
lets us apply Proposition \ref{ref} with $r=3$ and $a_j=\binom{-\rho}{j-1}$ to find that $\dm_{m,\ell}\left(\binom{-\rho}{3}, \binom{-\rho}{4},  \dots \right)$ equals
\begin{equation} \label{lab}
   \sum_{ j_1+ j_2+ j_3 +j_{4}= \ell}
 \binom{\ell}{j_1 , j_2 ,  j_3 , j_{4}}
   (-1)^{j_1+j_3} \rho^{j_2} \binom{-\rho}{2}^{j_3} \binom{ -\rho j_4}{m+2j_1+j_2+2j_4}.
\end{equation}

\begin{proof}[Proof of Proposition \ref{wrfo}]
Using \e{lab} in \e{qkr2} and reordering the summation to use the Chu-Vandermonde identity
\begin{equation*}
  \sum_{v=0}^{\infty} \binom{-\beta}{v} \binom{ -\rho j_4}{2r-v+2j_1+j_2+2j_4} = \binom{ -\rho j_4 -\beta}{2r+2j_1+j_2+2j_4},
\end{equation*}
leads to
\begin{multline} \label{lab2}
  \frac{\mathcal W_r(\rho,\beta)}{r! \binom{-\frac 12}{r}} = \sum_{\ell=0}^{2r} \binom{-\frac 12-r}{\ell} \\
  \times
   \sum_{ j_1+ j_2+ j_3 +j_{4}= \ell}
 \binom{\ell}{j_1 , j_2 ,  j_3 , j_{4}}
   (-1)^{j_1+j_3} \rho^{j_2} \binom{-\rho}{2}^{j_3-r-\ell} \binom{ -\rho j_4 -\beta}{2r+2j_1+j_2+2j_4}.
\end{multline}
Note that  the upper limit of the sum over $v$ in \e{qkr2} may be increased to $\infty$ because $\dm_{2r-v,\ell}$ is zero for $v>2r$. For simplicity, with $z\in \C$ and $j_1+ j_2+ j_3 +j_{4}= \ell$, we can write
\begin{equation*}
  \binom{z}{j_1 , j_2 ,  j_3 , j_{4}} := \frac{z(z-1) \cdots (z- \ell+1)}{j_1! j_2! j_3! j_{4}!} = \binom{z}{\ell} \binom{\ell}{j_1 , j_2 ,  j_3 , j_{4}}.
\end{equation*}
Hence \e{lab2} is
\begin{equation*}
  \sum_{ j_1+ j_2+ j_3 +j_{4}\lqs 2r}
 \binom{-\frac 12-r}{j_1 , j_2 ,  j_3 , j_{4}}
   (-1)^{j_1+j_3} \rho^{j_2} \binom{-\rho}{2}^{-r-j_1-j_2-j_4} \binom{ -\rho j_4 -\beta}{2r+2j_1+j_2+2j_4}.
\end{equation*}
One further simplification comes from summing over $j_3$:
\begin{align}
  \sum_{j_3=0}^{2r-j_1-j_2-j_4}\binom{z}{j_1 , j_2 ,  j_3 , j_{4}} (-1)^{j_3} &
  =\binom{z}{j_1 , j_2  , j_{4}}
  \sum_{j_3=0}^{2r-j_1-j_2-j_4}\binom{z-j_1 - j_2    - j_{4}}{j_3} (-1)^{j_3} \notag\\
   & = \binom{z}{j_1 , j_2  , j_{4}}
   (-1)^{j_1+j_2+j_4} \binom{z-j_1 - j_2    - j_{4}-1}{2r-j_1-j_2-j_4}\notag\\
   & = \binom{z}{j_1 , j_2  , j_{4}} \binom{2r-z}{2r-j_1-j_2-j_4}. \label{poy}
\end{align}
This made use of the identity
\begin{equation} \label{knuid}
  \sum_{\ell =0}^m (-1)^{\ell} \binom{z}{\ell} = (-1)^{m} \binom{z-1}{m} \qquad (m\in \Z, z\in \C),
\end{equation}
from \cite[(5.16)]{Knu}, and the basic relation $\binom{z}{k} = (-1)^k \binom{k-z-1}{k}$. Inserting \e{poy} with $z=-1/2-r$ and relabeling $j_4$ as $j_3$ finishes the proof.
\end{proof}

\section{Further asymptotics} \label{ss4}
To reduce the notation, set 
$\mathcal Q^*_r(k) :=\mathcal Q_r(k)/(k a_k)^r$.
If we define
\begin{equation} \label{qkn}
  q_{k,R}(n):= \mathcal M_k(n)\left( 1+\sum_{r=1}^{R-1} \frac{\mathcal Q^*_r(k)}{ (n+h_k)^{r/(k+1)}} \right),
\end{equation}
then Theorem \ref{wri} implies
\begin{equation} \label{qkn2}
  \frac{p^k(n)}{\mathcal M_k(n)}=\frac{q_{k,R}(n)}{\mathcal M_k(n)} + \bo{n^{-R/(k+1)}},
\end{equation}
 for positive integers $n$ and $R$. Now $q_{k,R}(n)$ makes sense for all values of $n$ except $n=-h_k$ and we can consider $q_{k,R}(n+\delta)$ for real $n$ and $\delta$ with $n\to \infty$ and $\delta$ fixed. Recall the definition of $\mathcal M_k(n)$ in \e{mkn}.

\begin{lemma} \label{fra}
For $n>0$ large enough we have the expansions
\begin{align}
  \exp\left( \alpha (n+\delta)^{1/(k+1)}-\alpha \cdot n^{1/(k+1)}\right) & =  \sum_{m=0}^{R-1} \frac{C_1(m,\delta)}{n^{m/(k+1)}}+\bo{\frac 1{n^{R/(k+1)}}}, \label{ca1}\\
  \frac{n^{3/2-1/(k+1)}}{(n+\delta)^{3/2-1/(k+1)}} & = \sum_{m=0}^{R-1} \frac{C_2(m,\delta)}{n^{m/(k+1)}}+\bo{\frac 1{n^{R/(k+1)}}},\label{ca2}\\
  \sum_{r=0}^{R-1} \frac{\mathcal Q^*_k(r)}{(n+\delta)^{r/(k+1)}} & =\sum_{m=0}^{R-1} \frac{C_3(m,\delta)}{n^{m/(k+1)}}+\bo{\frac 1{n^{R/(k+1)}}},\label{ca3}
\end{align}
for implied constants independent of $n$, where
\begin{align}
  C_1(m,\delta) & = \sum_{m/(k+1)\lqs \ell \lqs m/k} \frac{\alpha^{(k+1)\ell-m} \delta^\ell}{((k+1)\ell-m)!}
  \dm_{\ell,(k+1)\ell-m}\left(\binom{\frac 1{k+1}}{1}, \binom{\frac 1{k+1}}{2},  \dots \right), \label{c1}\\
  C_2(m,\delta) & = \begin{cases}
  \binom{-3/2+1/(k+1)}{m/(k+1)} \delta^{m/(k+1)} & \text{if} \quad (k+1) | m;\\
  0 & \text{if} \quad (k+1) \nmid m,
  \end{cases}
  \label{c2} \\
  C_3(m,\delta) & = \sum_{0\lqs \ell \lqs m/(k+1)} \binom{\ell-m/(k+1)}{\ell} \delta^\ell
  \mathcal Q^*_{m-(k+1)\ell}(k). \label{c3}
\end{align}
The indices $\ell$ in \e{c1} and \e{c3} run over all integers satisfying the given inequalities.
\end{lemma}
\begin{proof}
Writing $w=n^{-1/(k+1)}$, we find on the left of \e{ca1} the function
\begin{equation*}
  \psi(w,\delta):=\frac 1w \left(\left(1+\delta w^{k+1} \right)^{1/(k+1)}-1 \right) = \sum_{r=1}^\infty \binom{\frac 1{k+1}}{r}\delta^r w^{(k+1)r-1},
\end{equation*}
which is holomorphic for $w\in \C$ with $|w|<\delta^{1/(k+1)}$. Therefore $\exp(\alpha  \cdot \psi(w,\delta))$ is holomorphic on the same domain and, by the usual Taylor theorem with remainder,
\begin{equation}\label{drt}
  \exp(\alpha \cdot \psi(w,\delta)) = \sum_{m=0}^{R-1} C_1(m,\delta) w^m+\bo{|w|^R}.
\end{equation}
This shows that the form of \e{ca1} is correct. To compute the coefficients in \e{drt} write
\begin{align*}
  \exp(\alpha  \cdot \psi(w,\delta))  & = \sum_{j=0}^\infty \frac {\alpha^j}{j!} w^{-j} \left(\sum_{r=1}^\infty \binom{\frac 1{k+1}}{r} (\delta w^{(k+1)})^r \right)^j \\
   & =\sum_{j=0}^\infty \frac {\alpha^j}{j!} w^{-j}
   \sum_{\ell=j}^\infty (\delta w^{(k+1)})^\ell \dm_{\ell,j}\left(\binom{\frac 1{k+1}}{1}, \binom{\frac 1{k+1}}{2},  \dots \right)\\
   & =\sum_{m=0}^\infty w^m \sum_{\substack{0\lqs j\lqs \ell \\ (k+1)\ell-j=m}}
   \frac {\alpha^j}{j!} \delta^\ell \dm_{\ell,j}\left(\binom{\frac 1{k+1}}{1}, \binom{\frac 1{k+1}}{2},  \dots \right),
\end{align*}
and \e{c1} follows. The other parts are shown similarly using the binomial theorem.
\end{proof}

For $m=0$ we have $C_1(0,\delta)=C_2(0,\delta)=C_3(0,\delta)=1$. Fix $\alpha$ in $C_1(m,\delta)$ as $(k+1)a_k$ and define
\begin{equation} \label{ssk}
  \mathcal S_k(r,\delta):= \sum_{j_1+j_2+j_3=r} C_1(j_1,\delta) \cdot C_2(j_2,\delta) \cdot C_3(j_3,\delta).
\end{equation}
Then using Lemma \ref{fra} with $n$ replaced by $n+h_k$ in \e{qkn} easily gives:

\begin{cor} \label{roc}
For each fixed $\delta$, as $n \to \infty$,
\begin{equation} \label{siim}
  q_{k,R}(n+\delta) = \mathcal M_k(n) \left( 1+\sum_{r=1}^{R-1} \frac{\mathcal S_k(r,\delta)}{(n+h_k)^{r/(k+1)}} + O\left( \frac{1}{n^{R/(k+1)}}\right)\right)
\end{equation}
for an implied constant depending only on $k$, $R$ and $\delta$.
\end{cor}

The term $\mathcal M_k(n+\delta)$ has a similar expansion to \e{siim}, using just $C_1$ and $C_2$, and in particular we need
\begin{equation}\label{mmz}
  \mathcal M_k(n+\delta)/\mathcal M_k(n) = 1 +O(1/n^{1/(k+1)})
\end{equation}
as $n\to \infty$ for an implied constant depending only on $k$ and $\delta$.
Replacing $n$ by $n-h_k$ in Corollary \ref{roc} and letting $\delta = h_k$ means with \e{qkn2} and \e{mmz} that $p^k(n) \approx q_{k,R}(n) = q_{k,R}(n-h_k+\delta)$ and  we obtain the next result.

\begin{theorem} \label{ten}
Let $n$, $k$ and $R$ be positive integers. As $n \to \infty$,
\begin{equation} \label{aim2x}
  p^k(n) = b_k \frac{\exp\left((k+1) a_k  \cdot  n^{1/(k+1)} \right)}{n^{3/2-1/(k+1)}}\left( 1+\sum_{r=1}^{R-1} \frac{\mathcal S_k(r,h_k)}{n^{r/(k+1)}} + O\left( \frac{1}{n^{R/(k+1)}}\right)\right)
\end{equation}
for an implied constant depending only on $k$ and $R$.
\end{theorem}

Note that $C_1(m,0)=C_2(m,0)=\delta_{m,0}$ and $C_3(m,0)=\mathcal Q^*_m(k)$. It follows that $\mathcal S_k(r,0) = \mathcal Q^*_r(k)$.
As $h_k=0$ for $k$ even, we see that \e{aim2x} agrees with \e{aim} in this case. When $k$ is odd, \e{aim2x} gives a simpler asymptotic expansion for $p^k(n)$, in terms of $n$ instead of $n+h_k$, though this comes at the expense of requiring more complicated coefficients $\mathcal S_k(r,h_k)$. Theorem 1 of \cite{Tene19} is given in the form of Theorem \ref{ten} with the notation $\g_{kr}$ for $\mathcal S_k(r,h_k)$. We can compute, (see the proof of Theorem \ref{rat}),
\begin{equation*}
  \mathcal S_k(1,h_k) = \mathcal Q^*_1(k) = \mathcal Q_1(k)/(k a_k) \qquad (k \gqs 2),
\end{equation*}
agreeing with \cite{Tene19}, and also
\begin{equation*}
  \mathcal S_k(2,h_k) =  \mathcal Q_2(k)/(k a_k)^2 +\delta_{k,2} \cdot a_k h_k \qquad (k \gqs 2),
\end{equation*}
recalling \e{qqq}. 

In the $k=1$ case of Theorem \ref{ten} a simplification of the coefficients $\omega_r:=\mathcal S_1(r,-\frac 1{24})$ in \e{ssk} is possible. This is shown next and improves on \cite[Prop. 7.1]{odm}.

\begin{prop} \label{tenk1}
Let $n$ and $R$ be positive integers. As $n \to \infty$,
\begin{equation} \label{vu}
  p(n) = \frac{\exp\left( \pi \sqrt{2n/3} \right)}{4 \sqrt{3} n}\left( 1+\sum_{r=1}^{R-1} \frac{\omega_r}{n^{r/2}} + O\left( \frac{1}{n^{R/2}}\right)\right),
\end{equation}
with an implied constant depending only on  $R$, where
 \begin{equation} \label{red}
  \omega_r = \frac{1}{(-4\sqrt{6})^r} \sum_{k=0}^{(r+1)/2} \binom{r+1}{k} \frac{r+1-k}{(r+1-2k)!}  \left( \frac{\pi}6\right)^{r-2k}.
\end{equation}
\end{prop}
\begin{proof}
Let $c:=\pi \sqrt{2/3}$. By Theorem \ref{radq} we can write
\begin{equation*}
   p(n) = \frac{\exp\left( \pi \sqrt{2n/3} \right)}{4 \sqrt{3} n}\left(
  \frac{ \exp\left( c \sqrt{n}\left( \sqrt{1-\frac 1{24 n}} -1\right)  \right)}{1-\frac 1{24 n}}\left(1-\frac 1{ c \sqrt{n} \sqrt{1-\frac 1{24 n}} }\right)
    + O\left( \frac{1}{n^{R/2}}\right)\right).
\end{equation*}
For $z:=1/\sqrt{n}$, the above inner component is
\begin{equation*}
  \exp\left( \frac cz \left( \sqrt{1-\frac {z^2}{24}} -1\right)\right)
  \left(\frac 1{1-\frac {z^2}{24}} - \frac 1{\frac cz \left(1-\frac {z^2}{24}\right)^{3/2}}\right),
\end{equation*}
and, as a function of $z$, this is holomorphic in  a neighborhood of $z=0$ with Taylor expansion $1+\sum_{r=1}^{R-1} \omega_r z^{r} + O(|z|^R)$ for some coefficients $\omega_r$. This proves \e{vu}.

To find a formula for $\omega_r$, let $\alpha:=c \sqrt{n}$ and $x:=-1/(24 n)$ and rewrite the inner component as
\begin{equation} \label{szw}
  \exp\left( \alpha \left( \sqrt{1+x} -1\right)\right)
  \left(\frac 1{1+x} - \frac 1{\alpha \left(1+x\right)^{3/2}}\right) = \sum_{j=0}^\infty \xi_j(\alpha) x^j,
\end{equation}
which we will treat as a formal series in $x$. Integrating \e{szw} twice with respect to $x$ shows
\begin{align*}
  \xi_j(\alpha) & = (j+1)(j+2) \left[ x^{j+2}\right] \frac 4{\alpha^2} \exp\left( \alpha \left( \sqrt{1+x} -1\right)\right) \\
   &   = (j+1)(j+2) \left[ x^{j+2}\right] \frac 4{\alpha^2} \sum_{k= 0}^{j+2} \frac{\alpha^k}{k!} \left( \sqrt{1+x} -1\right)^k.
\end{align*}
 Let $y=\sqrt{1+x}$ and we will use
\e{genb} in the form
\begin{equation} \label{genbb}
  \left(\frac{1+y}2\right)^m = \left( \frac{1+(1+x)^{1/2}}2\right)^m = \mathcal B_{-1}(x/4)^m=
  \sum_{j=0}^\infty \frac m{m-j} \binom{m-j}{j} \frac {x^j}{4^j},
\end{equation}
which is valid for all $m\in \C$ when  $m \notin \Z_{\gqs 0}$. Then
\begin{align*}
  \xi_j(\alpha) & = (j+1)(j+2) \sum_{k= 0}^{j+2} \left(\frac {\alpha}2 \right)^{k-2}  \frac{1}{k!}  \left[ x^{j+2-k}\right] \left( \frac{2(\sqrt{1+x} -1)}{x}\right)^k \\
   &   = (j+1)(j+2) \sum_{k= 0}^{j+2} \left(\frac {\alpha}2 \right)^{k-2}  \frac{1}{k!}  \left[ x^{j+2-k}\right] \left( \frac{1+y}2\right)^{-k}\\
   & = (j+1) \sum_{k=0}^{j+1} \frac{(-1)^k \alpha^{j-k}}{2^{j+k} (j+1-k)!} \binom{j+k+1}{k},
\end{align*}
after simplifying. Since $\xi_j(\alpha) x^j$ in \e{szw} contains terms with factors $\alpha^{j-k} x^j =\frac{c^{j-k}}{(-24)^{j}} n^{-(j+k)/2}$,
\begin{equation*}
  \omega_r= \sum_{j+k=r, \ k\lqs j+1}   \frac{(-1)^k(j+1) }{2^{j+k} (j+1-k)!} \binom{j+k+1}{k}\frac{c^{j-k}}{(-24)^{j}}
\end{equation*}
and this reduces to \e{red}.
\end{proof}

A further notable asymptotic expansion of $p(n)$ is  given by Brassesco and  Meyroneinc in \cite{BM20} based on probabilistic methods. The natural parameter  in this context is
\begin{equation*}
  Y_n:= 1+2\left(\frac{2\pi^2}3 \left(n-\frac 1{24} \right)+\frac 14 \right)^{1/2},
\end{equation*}
and we may give a short proof of one of their main results, based on Theorem \ref{radq} (or \e{vuqr}).

\begin{prop}\label{pu} \cite[Prop. 2.2]{BM20}
As $n \to \infty$ we have
\begin{equation}\label{brass}
  p(n) = \frac{2\pi^2}{3\sqrt{3}} \frac{e^{(Y_n-1)/2}}{Y_n^2}\left( 1+\sum_{r=1}^{R-1} \frac{d_r}{Y_n^r} +\bo{\frac 1{Y_n^R}}\right),
\end{equation}
with an implied constant depending only on  $R$, for
\begin{equation}\label{drr2}
  d_r = \frac{r+1}{(-4)^r} \sum_{k=0}^{r+1}\binom{2r}{k} \frac{(-2)^k}{(r+1-k)!}.
\end{equation}
\end{prop}
\begin{proof} Substituting $Y_n/2 \sqrt{1-2/Y_n}$ for $\pi \sqrt{2/3(n-1/24)}$ in Theorem \ref{radq} produces
\begin{align*}
   p(n) & = \frac{2\pi^2}{3\sqrt{3}} \frac{\exp\left(\frac{Y_n}2 \sqrt{1-\frac 2{Y_n}} \right)}{Y_n^2(1-2/Y_n)}\left( 1-\frac 2{Y_n \sqrt{1-2/Y_n}} +\bo{\frac 1{Y_n^R}}\right) \\
   & = \frac{2\pi^2}{3\sqrt{3}} \frac{e^{(Y_n-1)/2}}{Y_n^2}  \left(
    \frac{\exp\left(\frac{Y_n}2 \sqrt{1-\frac 2{Y_n}} - \frac{Y_n-1}2\right)}{1-\frac 2{Y_n}}
    \left(1- \frac 2{Y_n \sqrt{1-\frac 2{Y_n}}} \right)
    +\bo{\frac 1{Y_n^R}}\right).
\end{align*}
Write $x=-2/Y_n$ and the above component
\begin{equation} \label{hai}
  \exp\left(\frac{1+x/2-\sqrt{1+x}}{x} \right)\left(\frac 1{1+x}+\frac x{(1+x)^{3/2}} \right)
\end{equation}
is holomorphic in a neighborhood of $x=0$ with a Taylor expansion $1+\sum_{r=1}^{R-1} (-2)^{-r} d_r x^r +O(|x|^R)$ for some numbers $d_r$.
This proves \e{brass} and it only remains to find the formula for $d_r$.

Integrating \e{hai} finds
\begin{align*}
  d_r & =(-2)^r(r+1) [x^{r+1}] \exp\left(\frac{1+x/2-\sqrt{1+x}}{x} \right)
 \frac{4(1+x/2+\sqrt{1+x})}{\sqrt{1+x}}\\
 & =(-2)^r(r+1) [x^{r+1}] \sum_{k=0}^{r+1} \frac 4{k!}\left(\frac{1+x/2-\sqrt{1+x}}{x} \right)^k
 \frac{1+x/2+\sqrt{1+x}}{\sqrt{1+x}}.
\end{align*}
 We used the fact that  terms with $k\lqs r+1$ make the only contributions to the coefficient of $x^{r+1}$ since the series expansion of the argument of $\exp$ begins $x/8-x^2/16+\cdots$. Let $y=\sqrt{1+x}$ and we may again use our techniques from Proposition \ref{1-2} and \e{genb}. As in \e{any}, it is simpler to have an $x^2$ denominator and so
 \begin{align}
   d_r & = (-2)^r(r+1)  \sum_{k=0}^{r+1} \frac 4{k!}[x^{r+1-k}]\left(\frac{1+x/2-\sqrt{1+x}}{x^2} \right)^k
 \frac{1+x/2+\sqrt{1+x}}{\sqrt{1+x}} \notag\\
 & = (-2)^r(r+1)  \sum_{k=0}^{r+1} \frac {8^{1-k}}{k!}[x^{r+1-k}]\frac 1y \left(\frac{1+y}{2} \right)^{2-2k} \label{tg}\\
 & = (-2)^r(r+1) \sum_{k=0}^{r+1} \frac {8^{1-k}}{k!}\frac {4(r+2-k)}{3-2k} [x^{r+2-k}]\left(\frac{1+y}{2} \right)^{3-2k} \label{tg2},
 \end{align}
 where we integrated \e{tg} to get \e{tg2}. Use \e{genb} to find the desired coefficient in \e{tg2} and then simplify to obtain \e{drr2}.
\end{proof}

Comparing Propositions \ref{tenk1}, \ref{pu} and Theorem \ref{radq}, shows that Theorem  \ref{radq} is to be preferred as it is the simpler and more accurate result.

\section{Convexity and log-concavity} \label{xv}

With the help of  Corollary \ref{roc},  we consider in this section the ratio
\begin{equation*}
  \frac{p^k(n+\delta)}{p^k(n)} \approx \frac{q_{k,R}(n+\delta)}{q_{k,R}(n)} \approx
  \left( 1+\sum_{r=1}^{R-1} \frac{\mathcal S_k(r,\delta)}{(n+h_k)^{r/(k+1)}} \right)
  \left( 1+\sum_{r=1}^{R-1} \frac{\mathcal Q^*_r(k)}{(n+h_k)^{r/(k+1)}} \right)^{-1}.
\end{equation*}

\begin{lemma} \label{liv}
Fix a positive integer $R$. For $n>0$ large enough we have the expansion
\begin{equation}
  \left(\sum_{r=0}^{R-1} \frac{\mathcal Q^*_r(k)}{n^{r/(k+1)}}\right)^{-1}  =\sum_{m=0}^{R-1} \frac{C_4(m)}{n^{m/(k+1)}}+ O\left( \frac{1}{n^{R/(k+1)}}\right),\label{ca4}
\end{equation}
where
\begin{equation}
  C_4(m)  = \sum_{\ell=0}^m (-1)^\ell
  \dm_{m,\ell}\left( Q^*_1(k),  Q^*_2(k),  Q^*_3(k), \dots \right) \label{c4}
\end{equation}
 and the implied constant in \e{ca4} depends only on $k$ and $R$.
\end{lemma}
\begin{proof}
This is similar to the proof of Lemma \ref{fra}. See \cite[Prop.~3.2]{odm} for the De Moivre polynomial formula for the coefficients of the multiplicative inverse of a power series.
\end{proof}

Define
\begin{equation} \label{tc}
  \mathcal T_k(r,\delta):= \sum_{j_1+j_2+j_3+j_4=r} C_1(j_1,\delta) \cdot C_2(j_2,\delta) \cdot C_3(j_3,\delta)\cdot C_4(j_4).
\end{equation}
The next result follows from Corollary \ref{roc} and Lemma \ref{liv}.

\begin{cor} \label{roc2}
Let $R$ be a positive integer and $\delta$ a fixed real number. As $n \to \infty$,
\begin{equation*}
  \frac{q_{k,R}(n+\delta)}{q_{k,R}(n)} =  1+\sum_{r=1}^{R-1} \frac{\mathcal T_k(r,\delta)}{(n+h_k)^{r/(k+1)}} + O\left( \frac{1}{n^{R/(k+1)}}\right)
\end{equation*}
for an implied constant depending only on $k$, $R$ and $\delta$.
\end{cor}

We also have
\begin{cor} \label{rocdo}
Let $R$ be a positive integer and $\delta$ a fixed integer. As $n \to \infty$,
\begin{equation} \label{es}
  \frac{p^k(n+\delta)}{p^k(n)} = 1 + \sum_{m=1}^{R-1} \frac{\mathcal T_k(m,\delta)}{(n+h_k)^{m/(k+1)}} +O\left( \frac{1}{n^{R/(k+1)}}\right),
\end{equation}
for an implied constant depending only on $k$, $R$ and $\delta$.
\end{cor}
\begin{proof}
Use \e{qkn2} to show
\begin{multline*}
  \frac{p^k(n+\delta)}{p^k(n)}  = \frac{\mathcal M_k(n+\delta)}{\mathcal M_k(n)}
  \left(\frac{q_{k,R}(n+\delta)}{\mathcal M_k(n+\delta)} + O\left( \frac{1}{n^{R/(k+1)}}\right) \right)
  \left(\frac{q_{k,R}(n)}{\mathcal M_k(n)} + O\left( \frac{1}{n^{R/(k+1)}}\right) \right)^{-1}\\
    =
  \left(\frac{q_{k,R}(n+\delta)}{\mathcal M_k(n)} + O\left( \frac{\mathcal M_k(n+\delta)}{\mathcal M_k(n)}\frac{1}{n^{R/(k+1)}}\right) \right)
  \left(\frac{q_{k,R}(n)}{\mathcal M_k(n)}\left(1 + O\left( \frac{\mathcal M_k(n)}{q_{k,R}(n)}\frac{1}{n^{R/(k+1)}}\right)\right) \right)^{-1}.
\end{multline*}
The error terms simplify to $O(n^{-R/(k+1)})$ by \e{qkn}, \e{mmz} and hence
\begin{multline*}
  \frac{p^k(n+\delta)}{p^k(n)}  =
  \left(\frac{q_{k,R}(n+\delta)}{q_{k,R}(n)} + O\left( \frac{1}{n^{R/(k+1)}}\right) \right)
  \left(1 + O\left( \frac{1}{n^{R/(k+1)}}\right) \right)^{-1}\\
    =
  \frac{q_{k,R}(n+\delta)}{q_{k,R}(n)} + O\left( \frac{1}{n^{R/(k+1)}}\right),
\end{multline*}
with \e{es} now following from Corollary \ref{roc2}.
\end{proof}

One further definition is  required:
\begin{equation} \label{ffmm}
  F(m)=F_k(m):=-\frac{1}{k+1} \sum_{j=1}^{m-k-1} j \mathcal Q^*_j(k) C_4(m-k-1-j).
\end{equation}
\begin{theorem} \label{rat}
Let $k$ and $\delta$ be integers with $k\gqs 2$. As $n\to \infty$
\begin{equation*}
  \frac{p^k(n+\delta)}{p^k(n)} = 1 + \sum_{m=k}^{2k+2} \frac{\mathcal T_k(m,\delta)}{(n+h_k)^{m/(k+1)}} +O\left( \frac{1}{n^{(2k+3)/(k+1)}}\right)
\end{equation*}
where the implied constant depends  only on $k$ and $\delta$, and we have the evaluations
\begin{align*}
  \mathcal T_k(k,\delta) & = a_k \delta, \\
  \mathcal T_k(k+1,\delta) & = (\textstyle{\frac 1{k+1}-\frac 32 }) \delta,\\
  \mathcal T_k(m,\delta) & = F(m) \delta \qquad (k+2 \lqs m \lqs 2k-1),\\
  \mathcal T_k(2k,\delta) & = F(2k) \delta + \textstyle{\frac 12}a_k^2 \delta^2,\\
  \mathcal T_k(2k+1,\delta) & = F(2k+1) \delta + a_k (\textstyle{\frac 3{2(k+1)}-2} )\delta^2, \\
  \mathcal T_k(2k+2,\delta) & = F(2k+2) \delta +\left( a_k F(k+2)+\binom{\frac 1{k+1}-\frac 32}{2} \right)\delta^2,
\end{align*}
where $\textstyle{\frac 16} a_k^3 \delta^3$ must be added to the expression for $\mathcal T_k(2k+2,\delta)$ if $k=2$.
\end{theorem}
\begin{proof}
This proof examines the components of $\mathcal T_k(m,\delta)$ in \e{tc}. 
Recall \e{c1}, (with $\alpha=(k+1)a_k$), \e{c2}, \e{c3} and \e{c4}.
We have $$C_1(0,\delta)=C_2(0,\delta)=C_3(0,\delta)=C_4(0)=1.$$
For $1\lqs m \lqs 2k+1$ we have $C_1(m,\delta)=0$ except for
\begin{equation} \label{chu}
  C_1(k,\delta)=a_k \delta, \qquad C_1(2k,\delta)=a^2_k \delta^2/2, \qquad C_1(2k+1,\delta)=-k a_k \delta^2/(2(k+1)).
\end{equation}
Also $C_1(2k+2,\delta)=0$ except when $k=2$, in which case it is $a^3_k \delta^3/6$. For $1\lqs m \lqs 2k+2$ we have $C_2(m,\delta)=0$ except for
\begin{equation} \label{chu2}
  C_2(k+1,\delta)=(\textstyle{\frac 1{k+1}-\frac 32 }) \delta, \qquad C_2(2k+2,\delta)=\binom{\frac 1{k+1}-\frac 32}{2}\delta^2.
\end{equation}
Next, by \e{c3},
\begin{equation} \label{chu3}
   C_3(m,\delta)= \mathcal Q^*_m(k) + \begin{cases} 0 & \text{ if \quad $0\lqs m \lqs k+1$}, \\
 (1-\frac m{k+1})\mathcal Q^*_{m-k-1}(k) \delta  & \text{ if \quad $k+2\lqs m \lqs 2k+2$}.
 \end{cases}
\end{equation}
We claim that also
\begin{equation} \label{rud}
   \sum_{j_3+j_4=m} C_3(j_3,\delta)\cdot C_4(j_4)=  \begin{cases} \delta_{m,0} & \text{ if \quad $0\lqs m \lqs k+1$}, \\
 F(m) \delta  & \text{ if \quad $k+2\lqs m \lqs 2k+2$}.
 \end{cases}
\end{equation}
To see this note that for all $m\gqs 0$, $\sum_{j_3+j_4=m} \mathcal Q^*_{j_3}(k)\cdot C_4(j_4)=   \delta_{m,0}$ since a power series divided by itself is $1$. Then \e{chu3} implies \e{rud} when $m\lqs k+1$. For $k+2\lqs m \lqs 2k+2$, the left side of \e{rud} equals
\begin{multline*}
   \sum_{j_3+j_4=m} Q^*_{j_3}(k)\cdot C_4(j_4) +
  \sum_{j_3=k+2}^{m} (1-\frac {j_3}{k+1})\mathcal Q^*_{j_3-k-1}(k) \delta \cdot C_4(m-j_3) \\
    = \delta \sum_{j_3=k+1}^{m} (1-\frac {j_3}{k+1})\mathcal Q^*_{j_3-k-1}(k) \cdot C_4(m-j_3) =  \delta F(m),
\end{multline*}
and we have established the claim \e{rud}.

Now  $\mathcal T_k(m,\delta)$  can be computed for $1\lqs m \lqs 2k+2$    by looking at all possible summands in \e{tc} using \e{chu}, \e{chu2} and \e{rud}.
\end{proof}

The case $k=1$ is shown similarly, with some extra terms appearing. Recall that $a_1=\pi/\sqrt{6}$.
\begin{theorem} \label{rat2}
Let  $\delta$ be an integer. As $n\to \infty$
\begin{equation*}
  \frac{p^1(n+\delta)}{p^1(n)} = 1 + \sum_{m=1}^{4} \frac{\mathcal T_1(m,\delta)}{(n+h_1)^{m/2}} +O\left( \frac{1}{n^{5/2}}\right)
\end{equation*}
where the implied constant depends  only on  $\delta$ and we have
\begin{align*}
  \mathcal T_1(1,\delta) & = a_1 \delta, \\
  \mathcal T_1(2,\delta) & = -\delta + \textstyle{\frac 12}a_1^2 \delta^2,\\
  \mathcal T_1(3,\delta) & =  \textstyle{\frac 1{4 a_1}}\delta - \textstyle{\frac 54}a_1 \delta^2
  +\textstyle{\frac 16}a_1^3 \delta^3, \\
  \mathcal T_1(4,\delta) & = \textstyle{\frac 1{8 a_1^2}}\delta +\textstyle{\frac 54} \delta^2
  -\textstyle{\frac 34}a_1^2 \delta^3 + \textstyle{\frac 1{24}}a_1^4 \delta^4.
\end{align*}
\end{theorem}

The odd powers of $\delta$ cancel in the next corollary of Theorems \ref{rat} and \ref{rat2}.
\begin{cor} \label{xyh}
Let $k$ and $\delta$ be integers with $k\gqs 1$. As $n\to \infty$,
\begin{equation} \label{ul}
  \frac 12 \left(\frac{p^k(n+\delta)}{p^k(n)}+ \frac{p^k(n-\delta)}{p^k(n)}\right) = 1 +  \frac{\textstyle{\frac 1{2}}a_k^2 \delta^2}{(n+h_k)^{2-2/(k+1)}} -\frac{(2-\textstyle{\frac 3{2(k+1)}} )a_k \delta^2}{(n+h_k)^{2-1/(k+1)}}+O\left( \frac{1}{n^{2}}\right),
\end{equation}
where the implied constant depends only on $k$ and $\delta$.
\end{cor}

Since $\textstyle{\frac 1{2}}a_k^2 \delta^2$ in \e{ul} is $> 0$ when $\delta \neq 0$, it follows that
\begin{equation}\label{con}
  2p^k(n) \lqs p^k(n+\delta) + p^k(n-\delta)
\end{equation}
for all $n$ sufficiently large. Hence, with $\delta=1$, $p^k(n)$ is asymptotically convex for each fixed $k\gqs 1$ as $n \to \infty$. Conjecture 3.4 of \cite{Ul21} claims
\begin{equation}\label{con2}
  2p^k(n) \lqs \left( p^k(n+1) + p^k(n-1)\right)\left( 1-n^{-k}\right)
\end{equation}
for large enough $n$ depending on $k\gqs 2$. As the second term on the right of \e{ul} is greater than $n^{-2}$ for $n$ large enough, we obtain the stronger estimate \e{redx}.

We have the further easy consequence of Theorems \ref{rat} and \ref{rat2}:

\begin{cor} \label{xyh2}
Let $k$ and $\delta$ be integers with $k\gqs 1$. As $n\to \infty$,
\begin{equation} \label{ul2}
  \frac{p^k(n+\delta)p^k(n-\delta)}{p^k(n)^2} = 1
  -\frac{(1-\textstyle{\frac 1{k+1}} )a_k \delta^2}{(n+h_k)^{2-1/(k+1)}}
   +\frac{(\textstyle{\frac 32 -\frac 1{k+1}} ) \delta^2}{(n+h_k)^{2}}+O\left( \frac{1}{n^{2+1/(k+1)}}\right),
\end{equation}
where the implied constant depends only on $k$ and $\delta$.
\end{cor}

Since $-(1-\textstyle{\frac 1{k+1}} )a_k \delta^2$ in \e{ul2} is $< 0$ when $\delta \neq 0$, it follows that
\begin{equation}\label{logcon}
  p^k(n)^2 \gqs p^k(n+\delta) \cdot p^k(n-\delta)
\end{equation}
for all $n$ sufficiently large.
Hence, with $\delta=1$, $p^k(n)$ is asymptotically log-concave for each fixed $k\gqs 1$ as $n \to \infty$. Conjecture 3.5 of \cite{Ul21} has
\begin{equation}\label{logcon2}
  p^k(n)^2 \gqs  p^k(n+1) \cdot p^k(n-1) \cdot \left( 1+n^{-k}\right)
\end{equation}
for large enough $n$ depending on $k\gqs 2$. Since $-1$ times the second term on the right of \e{ul2} is greater than $n^{-2}$ for $n$ large enough, we obtain the improvement \e{redx2}. A similar result appears in
\cite[Eq.~(28)]{bb}, based on the main theorem of \cite{Gaf16}, though unfortunately it contains an error.

The case $k=1$ and $\delta=1$ of Corollary \ref{xyh2} shows
\begin{equation}\label{dew}
  \frac{p(n+1)p(n-1)}{p(n)^2} = 1
  -\frac{\pi}{2\sqrt{6}}\frac{1}{(n-\tfrac 1{24})^{3/2}}
   +\frac{1}{(n-\tfrac 1{24})^{2}}+O\left( \frac{1}{n^{5/2}}\right),
\end{equation}
and this implies that
\begin{equation}\label{dew2}
  \frac{p(n+1)p(n-1)}{p(n)^2}\left( 1+\frac{\pi}{2\sqrt{6}}\frac{1}{n^{3/2}}\right) > 1,
\end{equation}
for large enough $n$. Chen, Wang and Xie proved in \cite{Ch16}  that \e{dew2} is true for $n\gqs 45$, establishing a conjecture of DeSalvo and Pak.

We lastly note that the terms $F(m)$ from \e{ffmm} cancel in the proofs of Corollaries \ref{xyh} and \ref{xyh2}. Our formulas for $\mathcal Q_r(k)$ were not needed  and so these two corollaries ultimately follow just from Theorem \ref{wri}. Extending \e{ul} and \e{ul2} to include more terms does require values of $\mathcal Q_r(k)$.

\section{Further connections with the work of Wright} \label{conw}

For $\rho>0$ and $\beta$, $z\in \C$, Wright defined the function
\begin{equation} \label{phi2}
  \phi(z)=\phi(\rho,\beta;z):= \sum_{\ell=0}^\infty \frac{z^\ell}{\ell! \cdot \G(\ell \rho+\beta)}.
\end{equation}
This is an entire function of $z$, and the case $\rho=1$ corresponds to the $J$-Bessel function. See \cite{Wr33} for some further properties of $\phi(\rho,\beta;z)$, including the following integral representation:
\begin{equation} \label{phi}
  \phi(\rho,\beta;z) = \frac 1{2\pi i} \int_{\mathcal H} w^{-\beta} \exp\left( w+\frac{z}{w^{\rho}}\right) \, dw,
\end{equation}
where $\mathcal H$ is the usual Hankel contour from $-\infty$, running below the real line, circling the origin in a positive direction and then running above the real line back to $-\infty$.
Then \e{phi} follows very simply by substituting into \e{phi2} Hankel's formula
\begin{equation*}
  \frac 1{\G(s)}  = \frac 1{2\pi i} \int_{\mathcal H} e^w w^{-s}  \, dw
    \qquad (s \in \C).
\end{equation*}

Now the connection to
\begin{equation} \label{puti}
  I_N  := \frac 1{2\pi i} \int_{U-iU^{2/3}}^{U+iU^{2/3}}
  w^{-\beta} \exp\left( w+\frac{N}{w^{\rho}}\right)  \, dw,
\end{equation}
in Proposition \ref{inz} is clear. Recall that
$U=(\rho N)^{1/(\rho+1)}$ is the saddle-point of $\exp\left( w+N/w^{\rho}\right)$.
We may take the contour in \e{puti} to be part of $\mathcal H$, and define $\mathcal H^*$ to be the remaining part, so that
\begin{equation} \label{ward}
  \phi(\rho,\beta;N) =I_N + \frac 1{2\pi i} \int_{\mathcal H^*}
  w^{-\beta} \exp\left( w+\frac{N}{w^{\rho}}\right)  \, dw.
\end{equation}
As the integral over $\mathcal H^*$ in \e{ward} does not include the saddle-point $U$, we expect it to be relatively small. Proving this next will allow us to show the asymptotic expansion of $\phi(\rho,\beta;N)$ as $N \to \infty$.

\begin{lemma}
Let $y$ and $\rho$ be real with $|y|\lqs 1/2$ and $\rho \gqs 0$. Then
\begin{equation} \label{ptoj}
  \Re\left( (1+i y)^{-\rho}\right) \lqs 1-\psi(\rho) y^2
  \qquad \text{for} \qquad
  \psi(\rho) =
  \begin{cases}
   \tfrac{2}{5}\rho & \text{if \quad $0\lqs \rho \lqs 1$},\\
   \tfrac 25 & \text{if \quad $ 1 \lqs \rho$}.
  \end{cases}
\end{equation}
\end{lemma}
\begin{proof}
We have
\begin{align*}
  \Re\left( (1+i y)^{-\rho}\right)
   & \lqs (1+y^2)^{-\rho/2}\\
   & \lqs (1-\tfrac 45 y^2)^{\rho/2} =\sum_{j=0}^\infty \binom{\rho/2}{j} (-\tfrac 45)^j y^{2j}.
\end{align*}
Except for the first, the  terms in the above series are all $\lqs 0$ when $0\lqs \rho \lqs 1$. We obtain \e{ptoj} in this case by omitting the terms with  $j\gqs 2$. Since $(1-\tfrac 45 y^2)^{\rho/2}$ is decreasing in $\rho$,
\begin{equation*}
  \Re\left( (1+i y)^{-\rho}\right) \lqs (1-\tfrac 45 y^2)^{1/2} \lqs 1-\tfrac 25 y^2,
\end{equation*}
for $\rho\gqs 1$, where the second inequality may be verified by squaring both sides.
\end{proof}

\begin{prop} \label{ward2}
Fix $\rho>0$ and  $\beta \in \C$.  As real $N \to \infty$,
\begin{equation} \label{sky}
  \frac 1{2\pi i}\int_{\mathcal H^*}
  w^{-\beta} \exp\left( w+\frac{N}{w^{\rho}}\right)  \, dw
  \ll |U^{-\beta}|
  \exp\left((1+\tfrac 1{\rho})U \right) \cdot   \exp\left(-C U^{1/3} \right),
\end{equation}
with $C>0$ depending only on $\rho$, and the implied constant depending only on $\rho$ and $\beta$.
\end{prop}
\begin{proof}
We can let the top half of $\mathcal H^*$ follow the vertical line from $U+iU^{2/3}$ to $U+iU/2$,  the arc $\tfrac{\sqrt{5}}2 U e^{i t}$ for $\pi/6\lqs t \lqs \pi$, and then the real line from $-\tfrac{\sqrt{5}}2 U$ to $-\infty$. We will bound the integral on this contour; the contribution from the symmetric bottom half of $\mathcal H^*$ will be the same.  For $w=U+i y$ with $U^{2/3}\lqs y\lqs U/2$, the integrand is
\begin{align}
  & \ll   \left||w|^{-\beta}\right| \exp\left( U+\frac{N}{U^{\rho}}\left(1-\psi(\rho)\frac{ y^2}{U^2}\right)\right) \notag\\
 & \ll
  |U^{-\beta}| \exp\left( (1+\tfrac 1{\rho})U\right) \exp\left(-\tfrac 1\rho \psi(\rho) U^{1/3}\right). \label{vert}
\end{align}
The  integrand on the arc is
\begin{equation*}
 \ll  |U^{-\beta}| \exp\left(\tfrac{\sqrt{5}}2  U \cos(t)+ \frac{N}{U^{\rho}} \left(\tfrac{\sqrt{5}}2  \right)^{-\rho} \cos(\rho t)\right).
\end{equation*}
As this is decreasing with $t$, the same bound \e{vert} applies on the arc. Finally, for the horizontal piece,
\begin{align*}
  \frac 1{2\pi i} \int_{-\tfrac{\sqrt{5}}2 U}^{-\infty} w^{-\beta} \exp\left( w+\frac{N}{w^{\rho}}\right)\, dw
    & =
  \frac 1{2\pi}\int_{\tfrac{\sqrt{5}}2 U}^\infty x^{-\beta} e^{-\beta i\pi} \exp\left( -x+\frac{N}{x^{\rho}e^{\rho i\pi}}\right) \, dx  \\
    & \ll \int_U^\infty |x^{-\beta}|  \exp\left( -x+\frac{N}{x^{\rho}}\cos(\rho \pi)\right) \, dx\\
    & \ll \exp\left( -\tfrac 12 U+\frac{N}{U^{\rho}}\right) \int_U^\infty |x^{-\beta}|  e^{-x/2} \, dx \\
    & \ll  \exp\left( (1+\tfrac 1{\rho})U\right) \exp\left(-\tfrac 32U\right),
\end{align*}
and this is  smaller than we need.
\end{proof}

So the integral in Proposition \ref{ward2} fits inside the error term of Proposition \ref{inz}, and with \e{ward} we have proved the following result, a special case of Theorems 1 and 2 of \cite{Wr35}.

\begin{theorem} \label{big}
Fix $\rho >0$, $\beta\in \C$ and a positive integer $R$. Then as real $N \to \infty$,
\begin{equation}\label{wrgk}
  \phi(\rho,\beta;N) = \frac{U^{1/2-\beta}}{\sqrt{2\pi (\rho+1)}}
  \exp\left((1+\tfrac 1{\rho})U \right)
 \left( 1+\sum_{r=1}^{R-1} \frac{\rho^r \mathcal W_r(\rho,\beta)}{U^{r}}
 +\bo{\frac{1}{ U^{R}}}\right),
\end{equation}
for $U=(\rho N)^{1/(\rho+1)}$ and an implied constant depending only on $\rho$, $\beta$ and $R$.
\end{theorem}
For example, Zagier's function $H_2(x)$ in \cite[p. 14]{Za21} is $2x^3 \phi(2,4;x^2)$ and the asymptotics quoted there follow from \e{wrgk}, at least for $x$ real.  Wright in fact found the asymptotics of $\phi\left(\rho, \beta; z\right)$ as $z\in \C$ goes to infinity in any direction. See also \cite{Br17} where the asymptotics of $\phi(-k/(k+1),1;z)$ are required in  establishing a more detailed version of a conjecture of Andrews on partitions without $k$ consecutive part sizes; Wright  covered the case $-1<\rho<0$ needed there too.

Comparing the identical asymptotics   of Theorem \ref{big} in the case \e{ori} with Theorem \ref{wri} makes it  evident that our main theorem may be restated succinctly with Wright's function:

\begin{cor}
Let  $k$  be a positive integer and $T$ any positive real. As $n \to \infty$,
\begin{equation} \label{aimss}
  p^k(n) = \frac {\phi\left(\tfrac 1k,- \tfrac 12; k c_k (n+h_k)^{1/k}\right)}{(2\pi)^{k/2} (n+h_k)^{3/2}}\left( 1 + O\left( \frac{1}{n^{T}}\right)\right)
\end{equation}
for an implied constant depending only on $k$ and $T$.
\end{cor}

Theorem 1 of \cite{Wriii34} gives a much stronger version of \e{aimss} with the error replaced by  $O( e^{-\alpha_k n^{1/(k+1)}})$ for some $\alpha_k>0$. This corresponds to taking the first term in Wright's theory. Theorem 3 of \cite{Wriii34} shows that the absolute error can be made $O( e^{-\alpha_k n^{1/(k+1)}})$, for any $\alpha_k>0$, by taking enough further terms. The results in \cite{Wriii34} are shown by uncovering the elaborate structure of $G_k(q)$ near roots of unity. For this see also  Schoenfeld \cite{Sc44} and recent work of Zagier \cite{Za21}. We hope to return to these topics in a future  paper.

{\small \bibliography{pwr-bib} }

{\small 
\vskip 5mm
\noindent
\textsc{Dept. of Math, The CUNY Graduate Center, 365 Fifth Avenue, New York, NY 10016-4309, U.S.A.}

\noindent
{\em E-mail address:} \texttt{cosullivan@gc.cuny.edu}
}

\end{document}